\documentclass[12pt]{amsart}
\usepackage{amsmath, amssymb, amsthm, enumerate, tikz, hyperref}

\newtheorem{theorem}{Theorem}[section]
\newtheorem{proposition}[theorem]{Proposition}
\newtheorem{lemma}[theorem]{Lemma}
\newtheorem{corollary}[theorem]{Corollary}

\def\tab{\hspace{1.6em}}

\numberwithin{equation}{section}
\numberwithin{figure}{section}

\def\({\Big(}
\def\){\Big)}
\def\ep{\varepsilon}
\def\u{\overline u}

\def\G{\overline G}
\def\Lip{\text{Lip}}

\newcommand{\R}{\mathbb{R}}

\def\a{\textbf a}

\def\A{\textbf A}
\def\B{\textbf B}

\def\I{\textbf I}

\def\K{\textbf K}
\def\L{\textbf L}

\def\DT{\dom_{L^2}\triangle(\SG)}
\def\DI{\dom_{L^\infty}\triangle(\SG)}
\def\E{\mathcal E}

\def\T{\mathcal T}
\def\dom{\text{dom}}
\def\SG{\mathcal{SG}}

\begin{document}

\begin{abstract}

We study boundary value problems for the Laplacian on a domain $\Omega$ consisting of the left half of the Sierpinski Gasket ($\SG$), whose boundary is essentially a countable set of points $X$. For harmonic functions we give an explicit Poisson integral formula to recover the function from its boundary values, and characterize those that correspond to functions of finite energy. We give an explicit Dirichlet to Neumann map and show that it is invertible. We give an explicit description of the Dirichlet to Neumann spectra of the Laplacian with an exact count of the dimensions of eigenspaces. We compute the exact trace spaces on $X$ of the $L^2$ and $L^\infty$ domains of the Laplacian on $\SG$. In terms of the these trace spaces, we characterize the functions in the $L^2$ and $L^\infty$ domains of the Laplacian on $\Omega$ that extend to the corresponding domains on $\SG$, and give an explicit linear extension operator in terms of piecewise biharmonic functions. 
\end{abstract}

\title{Boundary Value Problems On A Half Sierpinski Gasket}
\author{Weilin Li}
\address{Malott Hall, Department of Mathematics, Cornell University, Ithaca, NY 14853}
\email{wl298@cornell.edu}
\address{4423 Mathematics Building, University of Maryland, College Park, MD 20742}
\email{wl298@math.umd.edu}

\author{Robert S. Strichartz}
\address{563 Malott Hall, Department of Mathematics, Cornell University, Ithaca, NY 14853}
\email{str@math.cornell.edu}
\thanks{The first author was supported in part by the National Science Foundation grant DMS-0739164.}
\thanks{The second author was supported in part by the National Science Foundation grant DMS-1162045.}
\date{}
\maketitle

\section{Introduction}

The Laplacian on the Sierpinski Gasket was first constructed as a generator of a stochastic process, analogous to Brownian motion, by Kusuoka \cite{kusuoka} and Goldstein \cite{goldstein}. An analytic method of constructing the Laplacian on the Sierpinski Gasket as a renormalized limit of graph Laplacians was later developed by Kigami \cite{kigami1}. With a well defined Laplacian, it is possible to study differential equations on the Sierpinski Gasket, although strictly speaking, these are not differential equations. 

Harmonic functions on the Sierpinski Gasket have been studied in detail and the Dirichlet problem on the entire gasket reduces to solving systems of linear equations and multiplying matrices. However, there has been little research into boundary value problems on bounded subsets of fractals, except for \cite{owen}, \cite{qiu} and \cite{strichartz2} that consider domains generated by horizontal cuts of the gasket. Hence we believe it is appropriate to begin our exploration by studying the Dirichlet problem on a boundary generated by a vertical cut along one of the symmetry lines of the gasket. This is the simplest example of a boundary given as a level set of a harmonic function. We hope our results give insight into more general techniques for solving the Dirichlet problem and other boundary value problems on more general domains.  

Most of our results are applications of Kigami's harmonic calculus on fractals to our half gasket. His theory includes many mathematical objects specific to the world of fractal analysis, such as renormalized graph energies, normal derivatives and renormalized graph Laplacians. We will present some notation as we proceed, but for precise definitions and known facts (in particular the results that we call \textbf{Proposition}), see textbooks \cite{kigami2} and \cite{strichartz}.

The Sierpinski Gasket, denoted $\SG$, is the unique nonempty compact set satisfying $\SG = \bigcup_{j=0}^2 F_j \SG$ where $F_j$ are contractive mappings given by $F_jx = (x + q_j)/2$ and $q_j$ are the vertices of an equilateral triangle. Following convention, the boundary of $\SG$ is defined to be $V_0 = \{q_0, q_1, q_2\}$. Hence boundary in our language differs from the standard topological definition of boundary. Using the mappings $F_j$, we can iteratively generate a set of vertices $V_m$ where $m$ depends on the number of times we apply $F_j$. From $V_m$, we can find a graph approximation $\Gamma_m$. See Figure \ref{Vm} for an illustration. Notice how the boundary points $\{q_j\}$ are oriented and we keep this orientation for the entire paper. 
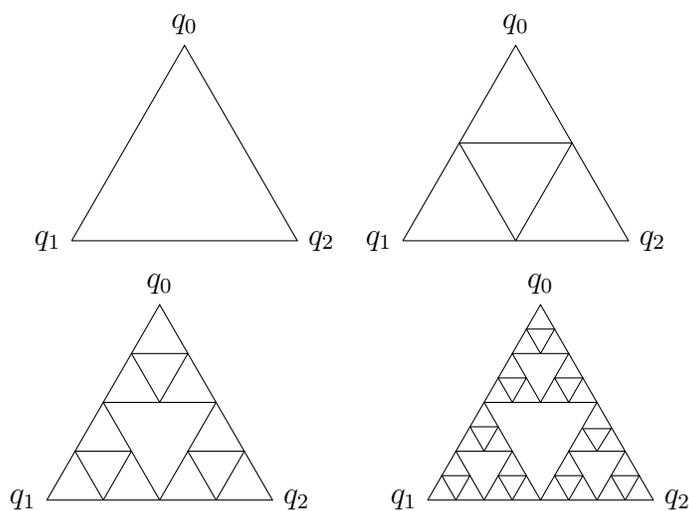
\begin{figure} [h]
\begin{center}
\begin{tikzpicture} [scale=0.5]
\draw 
(0, 0) node[anchor = east]{$q_1$} -- 
(6, 0) node[anchor = west]{$q_2$} -- 
(3, 5.2) node[anchor = south]{$q_0$} -- 
(0, 0);
\end{tikzpicture}
\indent
\begin{tikzpicture} [scale=0.5]
\draw 
(0, 0) node[anchor = east]{$q_1$} -- 
(6, 0) node[anchor = west]{$q_2$} -- 
(3, 5.2) node[anchor = south]{$q_0$} -- 
(0, 0);
\draw 
(3, 0) -- (1.5, 2.6) -- (4.5, 2.6) -- (3, 0);
\end{tikzpicture}
\tab
\begin{tikzpicture} [scale=0.5]
\draw 
(0, 0) node[anchor = east]{$q_1$} -- 
(6, 0) node[anchor = west]{$q_2$} -- 
(3, 5.2) node[anchor = south]{$q_0$} -- 
(0, 0);
\draw 
(3, 0) -- (1.5, 2.6) -- (4.5, 2.6) -- (3, 0);
\draw
(1.5, 0) -- (0.75, 1.3) -- (2.25, 1.3) -- (1.5, 0)
(4.5, 0) -- (5.25, 1.3) -- (3.75, 1.3) -- (4.5, 0)
(3, 2.6) -- (2.25, 3.9) -- (3.75, 3.9) -- (3, 2.6);
\end{tikzpicture}
\tab
\begin{tikzpicture} [scale=0.5]
\draw 
(0, 0) node[anchor = east]{$q_1$} -- 
(6, 0) node[anchor = west]{$q_2$} -- 
(3, 5.2) node[anchor = south]{$q_0$} -- 
(0, 0);
\draw 
(3, 0) -- (1.5, 2.6) -- (4.5, 2.6) -- (3, 0);
\draw
(1.5, 0) -- (0.75, 1.3) -- (2.25, 1.3) -- (1.5, 0)
(4.5, 0) -- (5.25, 1.3) -- (3.75, 1.3) -- (4.5, 0)
(3, 2.6) -- (2.25, 3.9) -- (3.75, 3.9) -- (3, 2.6);
\draw
(0.75, 0) -- (1.125, 0.65) -- (0.375, 0.65) -- (0.75, 0)
(2.25, 0) -- (1.875, 0.65) -- (2.625, 0.65) -- (2.25, 0)
(1.5, 1.3) -- (1.125, 1.95) -- (1.875, 1.95) -- (1.5, 1.3)
(3.75, 0) -- (4.125, 0.65) -- (3.375, 0.65) -- (3.75, 0)
(5.25, 0) -- (5.625, 0.65) -- (4.875, 0.65) -- (5.25, 0)
(4.5, 1.3) -- (4.125, 1.9) -- (4.875, 1.9) -- (4.5, 1.3)
(3.75, 2.6) -- (4.125, 3.25) -- (3.375, 3.25) -- (3.75, 2.6)
(2.25, 2.6) -- (2.625, 3.25) -- (1.875, 3.25) -- (2.25, 2.6)
(3, 3.9) -- (3.375, 4.55) -- (2.625, 4.55) -- (3, 3.9); 
\end{tikzpicture}
\end{center}
\caption{Left to right: $\Gamma_0, \Gamma_1, \Gamma_2, \Gamma_3$ of $\SG$}
\label{Vm}
\end{figure}

We work on the domain $\Omega$, which can be defined in terms of the level sets of a harmonic function. Let $h_s$ be the skew symmetric harmonic function with boundary values $(h_s(q_0), h_s(q_1), h_s(q_2)) = (0,1,-1)$. Then $\Omega = \{x \in \SG \setminus V_0 \colon h_s(x) > 0\}$ and $\partial \Omega = q_0 \cup q_1 \cup X$ where $X = \{ x \in \SG \setminus V_0 \colon h_s(x) = 0\}$. We write $\overline \Omega = \Omega \cup \partial\Omega$.

Figure \ref{omega} provides an illustration of $\overline \Omega$, which is precisely the left half of $\SG$ including the points on the symmetry line. In the figure, we labeled the points $x_m = F_0^{m-1} F_2 q_1$ and $y_m = F_0^m q_1$. Note that $X = \{x_m\}_{m=1}^\infty$, so each $x_m$ is important for obvious reasons. Each $y_m$ is important topologically because the removal of any $y_m$ turns $\Omega$ into a disconnected set. 

We also labeled the open sets $Y_m = F_0^{m-1} F_1(\SG \setminus V_0)$. Note that $\partial Y_m = \{x_m, y_{m-1}, y_m\}$ and we write $\overline Y_m = Y_m \cup \partial Y_m$. $\overline Y_m$ is classified as a cell because a cell is defined to be the image of $\SG$ under any compositions of contractive mappings $F_j$. Thus $\overline \Omega = \bigcup_{m=1}^\infty \overline Y_m$, which is an almost disjoint union. 

Although $\Omega$ is not globally self-similar because $\Omega$ cannot be written as a union of smaller copies of itself, it is locally self-similar because each $\overline Y_m$ is a fractal. The retention of this local property is extremely important for our analysis because any result regarding $\SG$ also holds for $\overline Y_m$ with a proper normalization factor. 
\begin{figure} [h]
\begin{center}
\begin{tikzpicture}
\draw [fill = lightgray]
(0, 0) node[anchor=east]{$q_1$} --
(3, 0) node[anchor=west]{$x_1$} --
(1.5, 2.6) -- 
(0, 0);
\draw [fill = lightgray]
(1.5, 2.6) node[anchor=east]{$y_1$} --
(3, 2.6) node[anchor=west]{$x_2$} --
(2.25, 3.9) --
(1.5, 2.6);
\draw [fill = lightgray]
(2.25, 3.9) node[anchor=east]{$y_2$} --
(3, 3.9) node[anchor=west]{$x_3$} --
(2.625, 4.55) --
(2.25, 3.9);
\draw [fill = lightgray]
(2.625, 4.55) node[anchor=east]{$y_3$} --
(3, 4.55) node[anchor=west]{$x_4$} --
(2.8125, 4.85) node[anchor=east]{$y_4$}--
(2.625, 4.55);
\fill[black] 
(3, 5.2) circle (1pt) node[anchor = south]{$q_0$};
\draw
(1.5, 1) node{\Huge $Y_1$}
(2.25, 3) node{\Large $Y_2$}
(2.625, 4.1) node{\tiny $Y_3$};
\end{tikzpicture}
\end{center}
\caption{A decomposition of $\overline \Omega$}
\label{omega}
\end{figure}
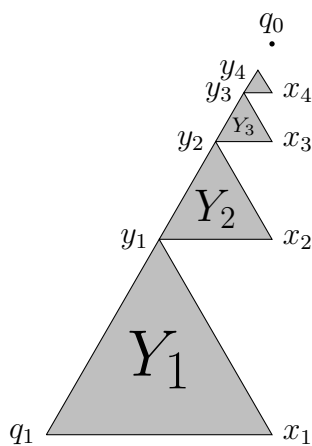

In the later sections, we will be interested in restriction and extension operators. Hence, we need to label points on the other half of the gasket. Let $z_m$ and $Z_m$ the reflections of $y_m$ and $Y_m$ respectively across the symmetry line containing $X$. Then $\SG = \bigcup_m (Y_m \cup Z_m)$ is an almost disjoint union and this decomposition will be useful in the later sections. 

We begin by studying the Dirichlet problem on $\Omega$:
\begin{equation} \label{bvp}
\begin{cases}
\ \triangle u = 0  		&\text{on } \Omega, \\
\ u(q_1) = a_0 			&\text{on } \partial \Omega, \\
\ u(x_m) = a_m 			&\text{on } \partial \Omega,
\end{cases}
\end{equation}
where $\triangle$ denotes the (Kigami) Laplacian with respect to the standard measure, $u: \overline \Omega \to \R$ is the unknown, and $\{a_m\}_{m=0}^\infty$ is the boundary data. Notice that we do not prescribe boundary data at $q_0$ even though $q_0 \in \partial \Omega$. This is by preference and is inconsequential because for almost the entire paper, we will assume $\{a_m\}$ converges. We will refer to (\ref{bvp}) as the BVP. 

In Section \ref{solution}, we construct a solution to the BVP using the harmonic extension algorithm, which we explain in that section. The space of $C(\Omega)$ solutions to the BVP is one-dimensional, but in general, the solutions blow up at $q_0$. We show that if the boundary data converges, then we can find a $C(\overline \Omega)$ solution that is unique in this function space.   

In Section \ref{energy}, we study the graph energy of the $C(\overline \Omega)$ solution to the BVP. Although its energy is complicated, the culminating theorem presents an equivalence between finite energy and the normalized summability of the the boundary data. In fact, finiteness depends only on how quickly the data converges and not on the limiting value. 

In Section \ref{normal}, we show that given stronger assumptions on the boundary data, we can obtain the existence of normal derivatives on $\partial \Omega$. In particular, we are interested in the behavior of the normal derivatives on $X$. The normal derivatives of the $C(\overline \Omega)$ solution on $X$ can be found in terms of the boundary data. This relationship allows us to define a Dirichlet to Neumann map and we show that this map is invertible.

In Section \ref{eigenfunctions}, we discuss both Dirichlet and Neumann eigenfunctions on $\Omega$. For more information on eigenvalues and eigenfunctions on fractals, see \cite{fukushima} and \cite{shima}. There are no new eigenfunctions on $\Omega$, but for a fixed eigenvalue, its multiplicity on $\Omega$ is different from its multiplicity on $\SG$. For each eigenfunction, we count the dimension of its eigenspace.
 
Section \ref{trace} and Section \ref{extension} are closely related to each other. We define a restriction operator that maps a function to its restriction to and normal derivatives on $X$. We characterize the function spaces $\DT$ and $\DI$ in terms of the restriction operator. Using this result, we provide necessary and sufficient conditions for extending functions in $\dom_{L^2}\triangle(\Omega)$ and $\dom_{L^\infty}\triangle(\Omega)$ to biharmonic functions in $\DT$ and $\DI$ respectively. 

Section \ref{appendix} acts as an appendix and in this section, we prove numerous lemmas about Green's functions and special types of sequences and series. Since these results are used in multiple sections and are purely technical lemmas, we have decided to place them in its own section. While the sequence and series lemmas may not be new, we have not found them in previously published work.  

It is important to mention that the results presented in this paper hold for any smaller copy of $\Omega$, $F_w(\Omega)$ for any word $w$, with different normalization constants.

\section{Solution to the Boundary Value Problem} \label{solution}

We begin this section by discussing the graph energy. The energy plays a central role in fractal analysis on $\SG$ because other objects such as harmonic functions, normal derivatives and the Laplacian, are defined in terms of the energy. Given a fixed value of $m$ and a real valued function $u$ on $\SG$, the (renormalized) graph energy of level $m$ is
\[
\E_m(u) = \sum_{x \sim_m y} \(\frac{5}{3}\)^m [u(x) - u(y)]^2,
\]
where $x \sim_m y$ means $x$ and $y$ are in the same cell of level $m$. The graph energy of $u$ is $\E(u) = \lim_{m \to \infty} \E_m (u)$, allowing the value $+\infty$. 

Given boundary conditions, we define a harmonic function to be the unique function that minimizes the graph energy subject these constraints. Additionally, our suggestive use of the word ``harmonic" is justified: harmonic functions as minimizers of energy are equivalent to functions that satisfy the differential equation $\triangle u = 0$. The Laplacian $\triangle$ is defined in Section \ref{normal}. 

The simplest tool for constructing harmonic functions subject to boundary conditions is the harmonic extension algorithm. For a function $u$ defined on $V_m$, we can define its harmonic extension to $V_{m+1}$ as follows. Let $\{v_j\}$ be the three boundary points of a cell with $\{u(v_j)\}$ given. Then the harmonic extension of $u$ to the three new points is shown in Figure \ref{harmonic extension}. It is not difficult to see that given $u$ on $V_m$, this is the unique extension that minimizes the graph energy at level $m+1$. 

We can apply the harmonic extension algorithm infinitely many times and the resulting function on $V_* = \bigcup_m V_m$ will be harmonic. It is not difficult to see that functions generated by the harmonic extension algorithm must be continuous. Furthermore, $V_*$ is dense in $\SG$ and so for continuous functions, it suffices to define them on a dense subset. Thus, we say a harmonic function is determined by its boundary values. 
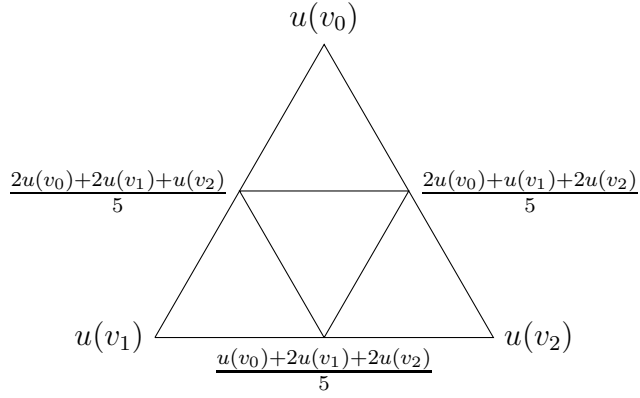
\begin{figure} [h]
\begin{center}
\begin{tikzpicture} [scale=0.75]
\draw 
(0, 0) node[anchor=east] {$u(v_1)$} -- 
(6, 0) node[anchor=west] {$u(v_2)$} -- 
(3, 5.2) node[anchor=south] {$u(v_0)$} -- 
(0, 0);
\draw 
(3, 0) node[anchor=north] {$\frac{u(v_0) + 2u(v_1) + 2u(v_2)}{5}$} -- 
(1.5, 2.6) node[anchor=east] {$\frac{2u(v_0) + 2u(v_1) + u(v_2)}{5}$} -- 
(4.5, 2.6) node[anchor=west] {$\frac{2u(v_0) + u(v_1) + 2u(v_2)}{5}$} -- 
(3, 0);
\end{tikzpicture}
\end{center}
\caption{Harmonic Extension Algorithm}
\label{harmonic extension}
\end{figure}

We can use the harmonic extension algorithm to construct a solution to the BVP. Any harmonic function on $\overline Y_m$ is determined by its values on $\partial Y_m$. Since $\overline \Omega = \bigcup_m \overline Y_m$, any harmonic function on $\overline \Omega$ is determined by its value at the points $\{x_m\}$ and $\{y_m\}$. In the following lemma, we see that there are additional constraints we must take into account.

\begin{lemma} \label{bvp L1}
Fix $m \geq 2$. Let $u$ be a continuous piecewise harmonic function with boundary data given by (\ref{bvp}). Then $\triangle u(y_m) = 0$ if and only if  
\begin{equation} \label{bvp recurrence}
u(y_m) = \frac{16}{5} u(y_{m-1}) - \frac{3}{5} u(y_{m-2}) - a_m - \frac{3}{5} a_{m-1}.
\end{equation} 
\end{lemma}

\begin{proof}
Consider the level $m$ approximation of $Y_{m-1} \cup Y_m$. The value of $u$ at the midpoint of $y_{m-1}$ and $y_{m-2}$ and the midpoint of $y_{m-1}$ and $x_{m-1}$ are determined by the harmonic extension algorithm, shown in Figure \ref{bvp F1}. 
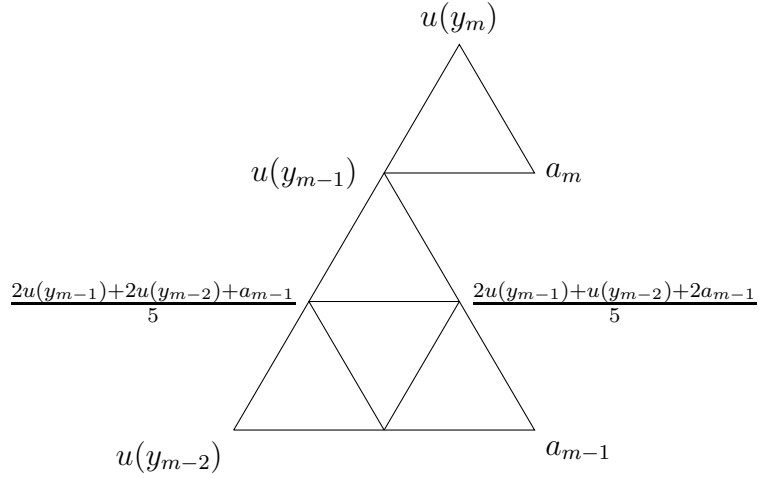
\begin{figure} [h]
\begin{center}
\begin{tikzpicture}
\draw 
(0,0) node[anchor = north east]{$u(y_{m-2})$} -- 
(4,0) node[anchor = north west]{$a_{m-1}$} -- 
(2,3.42) node[anchor = east]{$u(y_{m-1}) \hspace{0.5em}$} -- 
(0,0);
\draw 
(2,0) -- 
(3,1.71) node[anchor = east]{$\frac{2u(y_{m-1}) + 2u(y_{m-2}) + a_{m-1}}{5} \hspace{2cm}$} --
(1,1.71) node[anchor = west]{$\hspace{2cm} \frac{2u(y_{m-1}) + u(y_{m-2}) + 2a_{m-1}}{5}$} -- 
(2,0);
\draw (2,3.42) -- 
(4,3.42) node[anchor = west]{$a_m$} -- 
(3,5.13) node[anchor = south]{$u(y_m)$} -- 
(2,3.42);
\end{tikzpicture}
\end{center}
\caption{Harmonic extension}
\label{bvp F1}
\end{figure}
If $\triangle u(y_{m-1}) = 0$, then $u$ satisfies the mean value property at $y_{m-1}$. Thus, $u(y_{m-1})$ is the average of its four neighboring points in $V_m$ and simplifying that equation yields (\ref{bvp recurrence}). Conversely, if (\ref{bvp recurrence}) holds, then it is straightforward to check that $\triangle u(y_{m-1}) = 0$. 
\end{proof}

\begin{theorem} \label{bvp T1}
For every choice of convergent boundary data $\{a_m\}$, there is a one dimensional space of $C(\Omega)$ solutions to the BVP. Given a parameter $\lambda$, the solution to the BVP $u_\lambda$ is the harmonic extension of $u_\lambda(x_m) = a_m$, $u_\lambda(y_1) = \lambda$ and
\begin{align} \label{bvp b0}
u_\lambda(y_m) = 3^m F_m(\lambda) + \frac{1}{5^m} G_m(\lambda),
\end{align}
where
\[
F_m(\lambda)
= \frac{1}{14} \( 5\lambda - a_0 - a_1 - 18 \sum_{k=2}^m \frac{1}{3^k} a_k \)
\]
and 
\[
G_m(\lambda) 
= \frac{1}{14} \( - 5\lambda + 15 a_0 + 15a_1 + 4\sum_{k=2}^m 5^{k} a_k\).
\]
\end{theorem}

\begin{proof}
By Lemma \ref{bvp L1}, $u_\lambda$ must satisfy the recurrence (\ref{bvp recurrence}). The recurrence is linear, so we can formulate the equation in terms of matrices. Define 
\[
\A =
\begin{bmatrix} 
0 &0 \\ 
-\frac{3}{5} &-1 
\end{bmatrix}
\indent \text{and} \indent
\B =
\begin{bmatrix} 
0 &1 \\ 
-\frac{3}{5} & \frac{16}{5} 
\end{bmatrix}. 
\]
Then the recurrence can be written as  
\[
\begin{bmatrix}
u_\lambda(y_m) \\
u_\lambda(y_{m+1})
\end{bmatrix}
= \B^m 
\begin{bmatrix} 
a_0 \\ 
\lambda 
\end{bmatrix} 
+ \sum_{k =1}^m 
\B^{m-k} \A 
\begin{bmatrix}
a_k \\
a_{k+1}
\end{bmatrix}.
\]
Solving the system, we find that
\[
u_\lambda(y_m) 
= 3^m \(\frac{1}{14} \) \(5\lambda -a_0 - \sum_{k=1}^{m-1} \frac{1}{3^k} c_k\) + \frac{1}{5^m} \(\frac{1}{14} \) \(- 5\lambda + 15 a_0  + \sum_{k=1}^{m-1} 5^k c_k \),
\]
where $c_k = 5a_{k+1} + 3a_k$. We want our formula in terms of $a_k$ rather than $c_k$, so substituting  
\[
\sum_{k=1}^{m-1} 5^k c_k 
= 4 \sum_{k=2}^m 5^k a_k + 15a_1 - 5^m 3a_m
\]
and 
\[
\sum_{k=1}^{m-1} \frac{1}{3^k} c_k 
= 18 \sum_{k=2}^m \frac{1}{3^k} a_k + a_1 - \frac{1}{3^m} 3a_m
\]
into the previous equation for $u_\lambda(y_m)$ yields (\ref{bvp b0}). Extending these values by the harmonic extension algorithm uniquely yields a harmonic function $u_\lambda$ continuous on $\Omega$. 
\end{proof}

Since $u_\lambda$ is a linear combination of a $3^m$ term and a $1/5^m$ term, $u_\lambda$ may blow up at $q_0$. Naturally, we ask whether we can find a $\lambda$ such that $u_\lambda$ is continuous on $\overline \Omega$. 

\begin{lemma} \label{bvp L2}
Suppose $u_\lambda \in C(\Omega)$ satisfies the BVP for convergent $\{a_m\}$. Then $u_\lambda \in C(\overline \Omega)$ if and only if  
\begin{equation} \label{bvp cont}
\lim_{m \to \infty} u_\lambda(y_m) = \lim_{m \to \infty} u_\lambda(x_m).
\end{equation}
\end{lemma}

\begin{proof}
Suppose $u_\lambda \in C(\overline \Omega)$ solves the BVP. Then $u_\lambda$ is continuous at $q_0$, which implies (\ref{bvp cont}).  Conversely, it is easy to see that $q_0$ is the only point at which $u_\lambda$ can be discontinuous. Then (\ref{bvp cont}) implies $u_\lambda$ is continuous at $q_0$, which shows that $u_\lambda \in C(\overline \Omega)$. 
\end{proof}

\begin{theorem} \label{bvp T2}
If $a_m \to 0$ as $m \to \infty$, then the function $u$ given by the harmonic extension of $u(x_m) = a_m$, 
\begin{equation} \label{bvp b1}
u(y_1) = \frac{1}{5}\(a_0 + a_1 + 18 \sum_{k=2}^\infty \frac{1}{3^k} a_k\),
\end{equation}
and (for $m \geq 2$)
\begin{align} \label{bvp bm}
u(y_m) = \frac{1}{5^m} \( a_0 - \frac{9}{7} \sum_{k=1}^\infty \frac{1}{3^k} a_k + \frac{2}{7} \sum_{k=1}^m 5^k a_k \) + \frac{9}{7} \sum_{k=1}^\infty \frac{1}{3^k} a_{m+k}
\end{align}
solves the BVP. Furthermore, this function is the unique solution in  $C(\overline \Omega)$.
\end{theorem}

\begin{proof}
Substituting (\ref{bvp b1}) into (\ref{bvp b0}) yields (\ref{bvp bm}). By triangle inequality, 
\[
|u(y_m)| 
\leq 
\frac{1}{5^m} \(|a_0| + \frac{9}{7} \sum_{k=1}^\infty \frac{1}{3^k} |a_k| + \frac{2}{7} \sum_{k=1}^m 5^k |a_k| \) + \frac{9}{7} \sum_{k=1}^\infty \frac{1}{3^k} |a_{m+k}|.
\]
We claim that $|u(y_m)| \to 0$ as $m \to \infty$. Clearly the first term tends to zero in the limit. The second term tends to zero because convergent sequences are bounded. Since both the boundary data and $1/5^m$ converge to zero, for all $\ep > 0$, there exists $M$ such that for all $m \geq M$, we have $|a_m| < \ep$ and $1/5^m < \ep$. For $m \geq M$, we see that
\[
\sum_{k=1}^\infty \frac{1}{3^k} |a_{m+k}| 
\leq \ep \sum_{k=1}^\infty \frac{1}{3^k} 
= \frac{\ep}{2}
\]
and
\[
\frac{1}{5^m} \sum_{k=1}^m 5^k |a_k|
= \frac{1}{5^m} \sum_{k=1}^M 5^k |a_k| + \sum_{k=M+1}^m 5^{k-m} |a_k| 
\leq C_1 \ep \( \max_{1 \leq k \leq M} |a_k|\) + C_2 \ep.
\]
Therefore $u$ satisfies condition (\ref{bvp cont}) and by Lemma \ref{bvp L2}, $u \in C(\overline \Omega)$. Since harmonic functions that are continuous up to the boundary satisfy the maximum principle \cite{strichartz2}, uniqueness follows from the standard uniqueness argument for linear differential equations that satisfy the maximum principle.
\end{proof}

\begin{corollary} \label{bvp C1}
If $a_m \to A$ as $m \to \infty$ for some constant $A$, then the function $u$ given by the harmonic extension of $u(x_m) = a_m$, 
\begin{equation} \label{bvp u1}
u(y_1) = \frac{1}{5} \( a_0 + a_1 + 18 \sum_{k=2}^\infty \frac{1}{3^k} a_k \),
\end{equation}
and (for $m \geq 2$)
\begin{equation} \label{bvp um}
u(y_m) = \frac{1}{5^m} \( a_0 - \frac{9}{7} \sum_{k=1}^\infty \frac{1}{3^k} a_k + \frac{2}{7} \sum_{k=1}^m 5^k a_k \) + \frac{9}{7} \sum_{k=1}^\infty \frac{1}{3^k} a_{m+k}
\end{equation}
solves the BVP. Furthermore, this function is the unique solution in $C(\overline \Omega)$.  
\end{corollary}

\begin{proof}
Consider the modified BVP 
\begin{equation} \label{bvp modified}
\begin{cases}
\ \triangle u = 0 			&\text{on } \Omega, \\
\ u(q_1) = a_0 - A 		&\text{on } \partial \Omega, \\
\ u(x_m) = a_m - A 		&\text{on } \partial \Omega.
\end{cases}
\end{equation}
Since $a_m - A \to 0$, the hypotheses of Theorem \ref{bvp T2} are satisfied. Then there exists $w \in C(\overline \Omega)$ that solves (\ref{bvp modified}) and the formula for $w(y_m)$ is given by (\ref{bvp bm}) under the map $a_k \mapsto a_k - A$. By construction, the function $u = w + A$ solves the BVP with $u \in C(\overline \Omega)$. The maximum principle implies that $u$ is unique. 
\end{proof}

\section{Energy Estimate} \label{energy}

In this section, we look to answer questions regarding the energy of the $C(\overline \Omega)$ solution to the BVP. In particular, is the energy always finite and if not, can we characterize functions of finite energy in terms of a condition on the boundary data? Our main theorem shows that harmonic functions on $\Omega$ do not necessarily have finite energy and provides a simple characterization. 

Given a function $u$, we say $u \in \dom \E$ if and only if $\E(u) < \infty$. Following standard notation, $\dom_0\E$ is the space of functions that have finite energy and vanish on the boundary $V_0$. It is known that $\dom\E \subset C(\SG)$ and in fact, is a dense subset. 

Suppose $u$ is a piecewise harmonic function on $\Omega$ that is harmonic on each $Y_m$ with data given by (\ref{bvp}). Then the energy of $u$ restricted to $Y_m$ is constant after level $m$ and is determined by $u(y_m)$, $u(y_{m-1})$, and $a_m$. It follows that 
\[
\left. \E(u) \right|_{Y_m} 
= \(\frac{5}{3}\)^m [(u(y_m) - u(y_{m-1}))^2 + (u(y_m) - a_m)^2 + (u(y_{m-1}) - a_m)^2],
\]
where it is understood that $u(y_0) = u(q_1) = a_0$. Then $\E(u)$ is the sum of the energy of each cell,
\begin{equation} \label{energy eq1}
\E(u) 
= \sum_{m=1}^\infty \(\frac{5}{3}\)^m [(u(y_m) - u(y_{m-1}))^2 + (u(y_m) - a_m)^2 + (u(y_{m-1}) - a_m)^2].
\end{equation}
If we add the additional assumption that $u \in C(\overline \Omega)$ solves the BVP, then an equation for $\E(u)$ as a function of $\{a_m\}$ can be obtained by substituting (\ref{bvp u1}) and (\ref{bvp um}) into (\ref{energy eq1}). However, $\E(u)$ is series of quadratic terms of series, which is too complicated to analyze directly. Instead, we estimate it. 

\begin{lemma} \label{energy L1}
Suppose $u \in C(\overline \Omega)$ solves the BVP with convergent $\{a_m\}$. Then we have the energy estimate
\[
C_1 \sum_{m=1}^\infty \(\frac{5}{3}\)^m (a_{m+1} - a_m)^2  
\leq \E(u) \leq
C_2 \sum_{m=1}^\infty \(\frac{5}{3}\)^m (a_{m+1} - a_m)^2.
\]
\end{lemma}

\begin{proof} 
We prove the lower bound first. By ignoring the first term of (\ref{energy eq1}), we have 
\begin{align*}
\E(u) 
&\geq \sum_{m=1}^\infty \(\frac{5}{3}\)^m \left[ (u(y_m) - a_m)^2 + (u(y_{m-1}) -a_m)^2 \right] \\
&=\sum_{m=1}^\infty \(\frac{5}{3}\)^m (u(y_m)-a_m)^2+\sum_{m=0}^\infty \Big(\frac{5}{3}\Big)^{m+1}(u(y_m)-a_{m+1})^2.
\end{align*}
Using basic calculus, we find that $u(y_m) = (1/8)(5a_{m+1} + 3 a_m)$
minimizes the previous series. Substituting this value of $u(y_m)$ into the previous inequality, we obtain
\[ 
\sum_{m=1}^\infty \(\frac{5}{3}\)^m \frac{5}{8}(a_{m+1} -a_m)^2 + \frac{5}{3} (a_1 - a_0)^2 
\leq \E(u). 
\]
For the upper bound, consider the piecewise harmonic function $w$ given by the harmonic extension of $w(x_m)= w(y_m) = a_m$ and $w(q_1) = a_0$. Since $u$ is a global harmonic function while $w$ is a piecewise harmonic function, we have $\E(u) \leq \E(w)$. Note that $\E(w)$ is given by (\ref{energy eq1}) because $w$ is a piecewise harmonic function satisfying the boundary conditions. Then
\[ 
\E(u)
\leq \E(w) 
= \sum_{m=1}^\infty \(\frac{5}{3}\)^m \frac{10}{3}(a_{m+1} - a_m)^2 + \frac{10}{3} (a_1 - a_0)^2,
\]
which completes the proof.
\end{proof}

\begin{theorem} \label{energy T1}
Suppose $u \in C(\overline \Omega)$ solves the BVP with convergent boundary data $a_m \to A$. Then $u \in \dom \E$ if and only if $\|(5/3)^{m/2} (a_m-A)\|_{\ell^2} < \infty$. Additionally, we have the upper bound $\E(u) \leq C\|(5/3)^{m/2} (a_m-A)\|_{\ell^2}$.
\end{theorem}

\begin{proof}
Suppose $u \in C(\overline \Omega)$ solves the BVP with convergent boundary data $a_m \to A$. Lemma \ref{energy L1} says that $\E(u)<\infty$ if and only if $\|(5/3)^{m/2}(a_{m+1}-a_m)\|_{\ell^2} <\infty$. Applying Lemma \ref{series L2} yields the desired statement. 
\end{proof}

\section{Normal Derivatives} \label{normal}

Although the normal derivative and the (standard) Laplacian on $\SG$ are defined independently, they are closely connected via the Gauss-Green formula. 

For a continuous function $u$, its normal derivative at $q_j \in V_0$, denoted $\partial_n u(q_j)$, is defined to be
\begin{equation} \label{normal eq1}
\partial_n u(q_j) = \lim_{m \to \infty} \(\frac{5}{3}\)^m \left[ 2 u(q_j) - u(F_j^m q_{j+1}) - u(F_j^m q_{j-1}) \right]. 
\end{equation}
We say $\partial_n u(q_j)$ exists if the above limit exists. In the special case $u$ is harmonic, we have the simplified formula
\begin{equation} \label{normal eq2}
\partial_n u(q_j) = 2 u(q_j) - u(q_{j-1}) - u(q_{j+1}).
\end{equation} 
The formula for the normal derivative of a harmonic function at a boundary point of a cell is similar to the above formula, except we require a renormalization factor depending on the level. A junction point is a boundary point of two adjacent cells of the same level, and the normal derivative with respect to the cells will differ by a minus sign. If we need to distinguish between the two normal derivatives at a junction point, we use either $(\leftarrow, \rightarrow)$, $(\nearrow, \swarrow)$ or $(\nwarrow, \searrow)$, corresponding to the geometrical notion of a normal derivative. 

\begin{proposition}
Suppose $u \in \dom \triangle$. Then at each junction point, the local normal derivatives exist and $\nearrow \partial_n u \ + \swarrow \partial_n u = 0$. This is called the matching condition for normal derivatives.
\end{proposition}

The Laplacian of a function is defined in terms of its weak formulation. First, we define the (symmetric) bilinear form of the energy: given functions $u, v$ and integer $m$, the bilinear form of the energy is
\[
\E_m(u, v) = \sum_{x \sim_m y} \(\frac{5}{3}\)^m [u(x) - u(y)] [v(x) - v(y)].
\]
$\SG$ has a unique symmetric self-similar probability measure that we denote $dx$. Then the Laplacian can be defined as follows. Suppose $u \in \dom\E$ and $f$ is continuous. Then we say $u \in \dom \triangle$ with $\triangle u = f$ if
\[
\E(u,v) = - \int_\SG f(x) v(x) \ dx 
\]
for all $v \in \dom_0\E$ (functions in $\dom \E$ vanishing on $V_0$). Since $\E(u, v) = \E(v, u)$, subtracting the Gauss-Green formula from its transposed version yields the symmetric Gauss-Green formula
\begin{equation} 
\label{gauss-green}
\int_{\SG} (\triangle u v - u \triangle v) \ dx - \sum_{V_0} (v \partial_n u - u \partial_n v) = 0.
\end{equation}
The following result relates the normal derivatives of a function with its Laplacian.

\begin{proposition} [Gauss-Green]
Suppose $u \in \dom \triangle$. Then $\partial_n u$ exists on $V_0$ and the Gauss-Green formula, 
\[
\E(u, v) = - \int_{\SG} \triangle u v \ dx + \sum_{V_0} v \partial_n u,
\]
holds for all $v \in \dom \E$. 
\end{proposition}
 
For the remainder of this section, we assume $u \in C(\overline \Omega)$ solves the BVP with convergent boundary data. Naturally, we are interested in analyzing the behavior of $\partial_n u(x)$ for $x \in \partial \Omega$. For all points in $\overline \Omega$ except $q_0$, the formulas for the normal derivatives of $u$ are given by (\ref{normal eq2}). Using this equation, with the appropriate normalization factor, the normal derivative of $u$ at $y_m$ with respect to the cell $Y_m$ is
\begin{equation} \label{normal eq3}
\nearrow \partial_n u(y_m) = \(\frac{5}{3}\)^m [2 u(y_m) - u(y_{m-1}) - a_m].
\end{equation}
Similarly, the normal derivative of $u$ at $x_m$ with respect to $Y_m$ is 
\begin{equation} \label{normal eq4}
\rightarrow \partial_n u(x_m) = \(\frac{5}{3}\)^m [ 2a_m - u(y_m) - u(y_{m-1}) ].
\end{equation}
However (\ref{normal eq2}) does not give us the equation for $\uparrow \partial_n u(q_0)$ because $u$ is only defined on $\overline \Omega$. But we can define $\partial_n u(q_0)$ in a natural way. 

\begin{lemma}
If $u \in \dom\triangle (\SG)$, then 
\begin{equation} \label{normal eq5}
\uparrow \partial_n u(q_0) = 2 \cdot \lim_{m \to \infty} \nearrow \partial_n u(y_m).
\end{equation}
\end{lemma}

\begin{proof}
Write $u=u_s+u_a$, where $u_s$ and $u_a$ are the parts of $u$ that are symmetric and skew-symmetric with respect to $X$, respectively. Since $\left. u_a \right|_{F_0^m(SG)} = O(1/5^m)$, we have 
\[
\uparrow \partial_n u_a(q_0) 
= 2 \cdot \lim_{m \to \infty} \nearrow \partial_n u_a(y_m)
= 0.
\] 
For the symmetric part, consider the triangle $T_m$ with boundary points $\{q_0, y_m, z_m\}$ and the harmonic function $v$ on $T_m$ with $v(q_0)=v(y_m)=v(z_m)=1$. Applying the symmetric Gauss-Green formula (\ref{gauss-green}) for $u_s$ and $v$, we find that 
\[
\downarrow \partial_n u_s(q_0) \ + \nearrow \partial_n u_s(y_m) \ + \nwarrow \partial_n u_s(z_m) = \int_{T_m} \triangle u_s \ dx . 
\]
Notice that $\nearrow \partial_n u_s(y_m) = \ \nwarrow \partial_n u_s(z_m)$ by symmetry. Using the normal derivative matching condition of $u$ at $q_0$, we see that $\uparrow \partial_n u_s(y_m) = -\downarrow \partial_n u_s(q_0)$. Making these substitutions and taking the limit $m \to \infty$, we find that
\[
2 \cdot \lim_{m \to \infty} \nearrow \partial_n u_s(y_m) \ - \uparrow \partial_n u_s(q_0)
= \lim_{m \to \infty}\int_{T_m} \triangle u_s \ dx 
= 0,
\]
because $\triangle u$ is bounded and the measure of $T_m$ tends to zero in the limit. 
\end{proof}

Motivated by this lemma, we define $\uparrow \partial_n u(q_0)$ for $u$ defined on $\Omega$ by (\ref{normal eq5}). In the special case that $u\in C(\overline \Omega)$ solves the BVP with convergent data, then 
\begin{equation} \label{normal eq6}
\uparrow \partial_n u(q_0) 
= \lim_{m \to \infty} \left[ 5^m \(\frac{30}{7}\) \sum_{k=m+1}^\infty \frac{1}{3^k} a_k - \frac{1}{3^m} \(\frac{12}{7}\) \sum_{k=1}^{m} 5^k a_k \right],
\end{equation}
which we obtained by substituting (\ref{bvp um}) into the definition of $\uparrow \partial_n u(q_0)$.

Notice that (\ref{normal eq2}) implies that the normal derivatives of harmonic functions $\SG$ exist everywhere. However, this is not true for harmonic functions on $\Omega$ because the limit in (\ref{normal eq6}) may not exist. The following theorem characterizes when the limit exists. 

\begin{theorem} \label{new}
The normal derivative $\uparrow\partial_n u(q_0)$ exists if and only if the boundary data has the representation $a_m=A_1+A_2(3/5)^m+o((3/5)^m)$ for some constants $A_1$ and $A_2$. 
\end{theorem}

\begin{proof}
Suppose the limit (\ref{normal eq6}) exists. Let $b_m$ be the term in parentheses, and define $B=\lim_{m\to\infty} b_m$ and $c_m=(3/5)^{m+1} b_m$. A direct calculation shows that 
\[
35c_{m+2}-112c_{m+1}+21c_m=C(a_{m+2}-a_{m+1}), 
\]
where $C=-126$. This implies $a_m$ is dominated by a geometric series, hence it is a Cauchy sequence and converges to some limit $A_1$. Writing $a_m$ as a telescoping series, we have
\begin{align*}
A_1-a_m
=\sum_{k=m}^\infty (a_{k+1}-a_k) 
&=\frac{1}{C}\sum_{k=m}^\infty (35c_{k+2}-112c_{k+1}+21c_k) \\
&=\frac{1}{C}\sum_{k=m}^\infty \(\frac{3}{5}\)^{k+2}(21b_{k+2}-112b_{k+1}+35b_k).
\end{align*}
Let $A_2=(252/5)(B/C)$. Adding $A_2(3/5)^m=56(B/C)\sum_{k=m}^\infty (3/5)^{k+2}$ to both sides of the above equation, we find that
\[
A_1-a_m+A_2\(\frac{3}{5}\)^m=\frac{1}{C}\sum_{k=m}^\infty \(\frac{3}{5}\)^{k+2}[21(b_{k+2}-B)-112(b_{k+1}-B)+35(b_k-B)].
\]
Finally, taking the absolute value of both sides, we obtain 
\[
\left|a_m-A_1-A_2\(\frac{3}{5}\)^m\right|
\leq C'\sum_{k=m}^\infty \(\frac{3}{5}\)^{k+2}\(|b_{k+2}-B|+|b_{k+1}-B|+|b_k-B|\)
\]
Since $|b_k-B|\to 0$ as $k\to\infty$, we conclude that $a_m-A_1-A_2(3/5)^m=o((3/5)^m)$.

Conversely, if $a_m=A_1+A_2(3/5)^m+o((3/5)^m)$, then clearly the limit (\ref{normal eq6}) exists and equals a constant times $A_2$.
\end{proof}

To find the normal derivatives on $X$ in terms of the boundary data, we substitute (\ref{bvp um}) into (\ref{normal eq4}), which yields 
\begin{equation} \label{normal eq7}
\eta_m 
= \(\frac{5}{3}\)^m \( 3a_m - \frac{12}{7} \sum_{k=1}^\infty \frac{1}{3^k} a_{m+k} \) - \frac{1}{3^m} \( 6 a_0 + \frac{12}{7} \sum_{k=1}^m 5^k a_k - \frac{54}{7} \sum_{k=1}^\infty \frac{1}{3^k} a_k \),
\end{equation}
where $\eta_m = \ \rightarrow \partial_n u(x_m)$. We can think of (\ref{normal eq7}) as a Dirichlet to Neumann map on $X$ because it maps the Dirichlet boundary data to the corresponding normal derivatives. Define the infinite vectors
\[
\boldsymbol{\eta} = 
\begin{bmatrix} 
\eta_1 \\
\vdots \\
\eta_i \\
\vdots \\
\end{bmatrix}, 
\indent 
\a = 
\begin{bmatrix}
a_1 \\
\vdots \\
a_i \\
\vdots
\end{bmatrix}
\indent \text{and} \indent
\a_0 = 6a_0
\begin{bmatrix}
1/3 \\
\vdots \\
1/3^i \\
\vdots
\end{bmatrix},
\]
and the infinite matrices $\L = \text{Diag}[(5/3)^i]$ and $\K$ with entries 
\[
K_{i,j} = 
\begin{cases}
\vspace{0.5em}
\ \frac{7}{16} - \frac{27}{8} \frac{1}{5^i} \frac{1}{3^j} &\text{if } i = j, \\ 
\vspace{0.5em}
\ \frac{3}{4} \frac{3^i}{3^j} - \frac{27}{8} \frac{1}{5^i} \frac{1}{3^j} &\text{if } i < j, \\
\vspace{0.5em} 
\ \hspace{0.05em} \frac{3}{4} \frac{5^j}{5^i} -  \frac{27}{8} \frac{1}{5^i} \frac{1}{3^j} &\text{if } i > j.
\end{cases}
\]
Then (\ref{normal eq7}) can be written as 
\[
\boldsymbol \eta = \frac{16}{7} \L(\I-\K)\a + \a_0.
\]
Since we assumed $\{a_m\}$ converges and $u \in C(\overline \Omega)$, we see that $\{a_m\}, \{u(y_m)\} \in \ell^\infty$. Then (\ref{normal eq4}) implies $\|(3/5)^m \eta_m\|_{\ell^\infty} < \infty$. For this reason, for a real number $r$, we define the space 
\[
\ell^{r,\infty} = \{\{c_m\} \colon \|r^m c_m\|_{\ell^\infty} < \infty \}.
\] 
Then we define the Dirichlet to Neumann map $D_N \colon \ell^\infty \to \ell^{3/5, \infty}$ given by
\[
D_N \a = \frac{16}{7} \L(\I-\K)\a + \a_0.
\]

\begin{theorem} \label{tn1}
The Dirichlet to Neumann map is invertible.
\end{theorem}

\begin{proof}
We see that $D_N$ is a composition of $\L \colon \ell^\infty \to \ell^{3/5, \infty}$ with $\I-\K \colon \ell^\infty \to \ell^\infty$ plus a translation. The translation is not important and obviously $\L$ is invertible because it is diagonal. 

\noindent
It is well known that $\I-\K$ is invertible if and only if $\rho (\K) < 1$, where $\rho(\K)$ is the spectral radius of $\K$. The sum of the entries of the $i$-th row is
\[
\sum_{j=1}^\infty K_{i,j} 
= K_{i, i} + \sum_{j=1}^{i-1} K_{i,j} + \sum_{j=i+1}^{\infty} K_{i,j} 
< \frac{7}{16} + \frac{3}{4} \( \sum_{j=1}^{i-1} \frac{5^j}{5^i} + \sum_{j=i+1}^{\infty} \frac{3^i}{3^j}\). 
\]
Consequently, 
\[
\|\K\|_\infty 
= \sup_i \sum_{j=1}^\infty K_{i,j} 
< \frac{7}{16} + \frac{3}{4} \( \sum_{j=1}^{\infty} \frac{1}{5^j} + \sum_{j=1}^{\infty} \frac{1}{3^j} \)
= 1.
\]
Since $\K$ is a positive matrix, the Perron-Frobenius Theorem for positive matrices states that $\rho(\K) \leq \|\K\|_\infty$. Thus, $\rho(\K) < 1$, which shows that $\I - \K$ is invertible.
\end{proof}

\section{Eigenfunctions} \label{eigenfunctions}

The exact spectral asymptotics on the whole gasket and the structure of the spectrum has been analyzed previously \cite{strichartz3}. Motivated by that result, we discuss eigenvalues and eigenfunctions on the half gasket. Observe that:
\begin{enumerate}
\item
A Dirichlet eigenfunction on $\Omega$ extends by odd reflection to a Dirichlet eigenfunction on $\SG$ and conversely. 
\item
A Neumann eigenfunction on $\Omega$ extends by even reflection to a Neumann eigenfunction on $\SG$ and conversely.
\end{enumerate}
Thus there are no new eigenvalues on $\Omega$ because odd eigenfunctions on $\SG$ are Dirichlet eigenfunctions on $\Omega$ and even eigenfunctions on $\SG$ are Neumann eigenfunctions on $\Omega$. Hence we count the number of even and odd eigenfunctions on $\SG$. 

On $\SG$, there are $\# V_m = (3^{m+1} + 3)/2$ vertices on level $m$, of which $m+1$ lie on $q_0 \cup X$ and three are boundary points $V_0$. The eigenfunctions with eigenvalue $\lambda \leq C_0 5^m$ for a specific choice of $C_0$ are born on level $k \leq m$ and are in one-to-one correspondence with the graph eigenfunctions on $V_m$, so there are $(3^{m+1} +3)/2$ Neumann eigenfunctions and $(3^{m+1}- 3)/2$ Dirichlet eigenfunctions. Thus on $\Omega$, 
\begin{align*}
\# \{\text{Neumann eigenfunctions with } \lambda \leq C_0 5^m\} 
&= \frac{1}{2} \(\frac{3^{m+1} + 3}{2} + m + 1\), \\
\# \{\text{Dirichlet eigenfunctions with } \lambda \leq C_0 5^m\} 
&= \frac{1}{2} \(\frac{3^{m+1} - 3}{2} - m\),
\end{align*}
because the $m+1$ vertices on $q_0 \cup X$ contribute even functions to the Neumann count while the $m$ vertices on $X$ do not contribute to the Dirichlet count. Note that the correction terms $m+1$ and $-m$ are of the order $\log 5^m$. This is consistent with the observation that $\partial \Omega$ is zero dimensional. We can be more specific about individual multiplicities of eigenvalues on $\Omega$. For a set $U$, define the functions 
\begin{align*}
N(U) &= \# \{\text{Neumann eigenfunctions on } U\}, \\
D(U) &= \# \{\text{Dirichlet eigenfunctions on } U\}.
\end{align*}
Each eigenfunction born on level $k$ restricts to a graph eigenfunction on $V_k$ with eigenvalue $\mu$ with $\mu=0,2,3,5,$ or $6$. We say that the eigenfunction belongs to the $\mu$-series. This is explained in detail in \cite{strichartz} and \cite{strichartz3}, together with bifurcation rules that explain how the restriction of the eigenfunction to $V_k$ leads to several different eigenfunctions on $\SG$. The multiplicity of the eigenspaces only depends on $k$ and $\mu$ and is explicitly computed on $\Omega$ as follows. 
\begin{enumerate}
\item 
0-series (constant eigenfunctions) have multiplicity $N(\Omega) = 1$ and $D(\Omega) = 0$. 
\item
2-series only show up in the Dirichlet spectrum on $\SG$, but they are all even so they are absent from the Dirichlet spectrum of $\Omega$. Thus, $N(\Omega) = 0$ and $D(\Omega) = 0$.
\item
3-series are entirely Neumann eigenfunctions on $\SG$ that are born on level 0 with multiplicity 2. Then $N(\Omega) = 1$ and $D(\Omega) = 0$. 
\item
5-series are born on level $k$ where $k \geq 1$ for Dirichlet eigenfunctions and $k \geq 2$ for Neumann eigenfunctions. If $S_k$ denotes the number of cycles of level less than $k$, then on $\SG$, we find that $N(\SG) = S_k$ and $D(\SG) = S_k + 2$. For a cycle that lies on $X$, the eigenfunction is odd, so that contributes to $D(\SG)$ but not to $N(\SG)$. See Figure \ref{eigenfunction F1} for an example of such a function. Note that any unlabeled point means the function is defined to be zero at that point. Additionally, of the two extra Dirichlet eigenfunctions on $\SG$, exactly one is odd, as shown in Figure \ref{eigenfunction F2}. 

\begin{figure}[h]
\begin{center}
\begin{tikzpicture} [scale = 0.75]
\draw 
(0, 0) -- 
(6, 0) -- 
(3, 5.2) -- 
(0, 0);
\draw 
(3, 0) -- 
(1.5, 2.6) -- 
(4.5, 2.6) -- 
(3, 0);
\draw
(1.5, 0) node[anchor=north]{$-1$} -- 
(0.75, 1.3) node[anchor=east]{$1$} -- 
(2.25, 1.3) -- 
(1.5, 0);
\draw
(4.5, 0) node[anchor=north]{$1$} -- 
(5.25, 1.3) node[anchor=west]{$-1$} -- 
(3.75, 1.3) -- 
(4.5, 0);
\draw
(3, 2.6) -- 
(2.25, 3.9) node[anchor=east]{$-1$} -- 
(3.75, 3.9) node[anchor=west]{$1$} -- 
(3, 2.6);
\end{tikzpicture}
\end{center}
\caption{Odd eigenfunction on $\Gamma_2$}
\label{eigenfunction F1}
\end{figure}
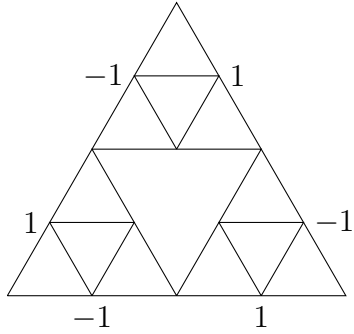

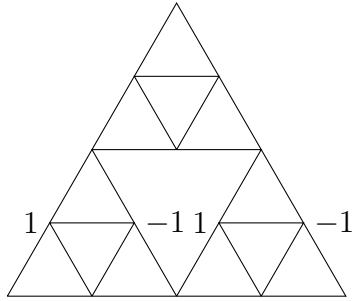
\begin{figure} [h]
\begin{center}
\begin{tikzpicture} [scale = 0.75]
\draw 
(0, 0) -- 
(6, 0) -- 
(3, 5.2) -- 
(0, 0);
\draw 
(3, 0) -- 
(1.5, 2.6) -- 
(4.5, 2.6) -- 
(3, 0);
\draw
(1.5, 0) -- 
(0.75, 1.3) node[anchor=east]{$1$} -- 
(2.25, 1.3) node[anchor=west]{ $-1$} -- 
(1.5, 0);
\draw
(4.5, 0) -- 
(5.25, 1.3) node[anchor=west]{$-1$} -- 
(3.75, 1.3) node[anchor=east]{$1$} -- 
(4.5, 0);
\draw
(3, 2.6) -- 
(2.25, 3.9) -- 
(3.75, 3.9) -- 
(3, 2.6);
\end{tikzpicture}
\end{center}
\caption{Another odd eigenfunction on $\Gamma_2$}
\label{eigenfunction F2}
\end{figure}

\medskip \noindent
The number of cycles of level $n$ is $3^{n-1}$ and exactly one of these lies on $X$. 
So there are $(3^{n-1} + 1)/2$ odd eigenfunctions and $(3^{n-1} - 1)/2$ even eigenfunctions. Thus
\[
N(\Omega)=\sum_{n=1}^{k-1}\frac{1}{2}(3^{n-1}-1)=\frac{1}{2}\(\frac{3^{k-1}+1}{2}-k\)
\]
and
\[
D(\Omega)=\(\sum_{n=1}^{k-1}\frac{1}{2}(3^{n-1}+1)\)+1=\frac{1}{2}\(\frac{3^{k-1}+1}{2}+k\).
\]
\item
6-series on $\SG$ are born on level $k$ where $k \geq 1$ for Neumann eigenfunctions and $k \geq 2$ for Dirichlet eigenfunctions. We know that $N(\SG) = \#V_{k-1}$ and $D(\SG) = \#V_{k-1} - 3$. Neumann eigenfunctions are obtained by giving arbitrary values on the points in $V_{k-1}$, while Dirichlet eigenfunctions are obtained by giving arbitrary values on the points $V_{k-1} \setminus V_0$. 

\medskip\noindent
To find the multiplicities on $\Omega$, we just have to count the even eigenfunctions and the odd eigenfunctions. Hence  
\[
N(\Omega) = \frac{1}{2} \(\frac{3^k + 3}{2} + k\) 
\indent \text{and} \indent 
D(\Omega) = \frac{1}{2} \(\frac{3^k - 3}{2} - k+1\).
\]
\end{enumerate}

\section{Trace Theorem} \label{trace}

Consider the restriction map $R$ given by $Ru = \{(u(x_m),\partial_nu(x_m))\}$, where $u$ is some function defined on some set containing $X$. That is, $R$ maps $u$ to its function values on $X$ and its normal derivatives on $X$. In this section, we determine the image of $\DT$ and $\DI$ under $R$. We say that $u \in \DT$ if $u$ is continuous on $\SG$ and $\triangle u \in L^2(\SG)$, and analogously for $u \in \DI$.

To simplify notation, we define the following spaces. Define the Lipschitz space 
\[
\Lip = \{\{c_m\} \colon \text{there exists } M \text{ such that } |c_{m+1}-c_m| \leq M \text{ for all } m \}.
\]
The norm on Lip/Constants is $\|c_m\|_{\Lip} = \inf M$ where the infimum is taken over all $M$ satisfying the previous condition. It follows directly from the definition of $\Lip$ that $\{c_m\} \in \Lip$ if and only if there exists $M$ such that $|c_m - c_n| \leq M |m-n|$ for all $m$ and $n$. 

We define the following trace spaces:
\begin{align*}
\T_\infty 
&= \big\{ \{(a_m,\eta_m)\} \colon 
a_m = A_1 + A_2(3/5)^m + a_m',  \\
&\hspace{10em}\|5^m a_m'\|_{\ell^\infty}<\infty, \|3^m\eta_m\|_\Lip <\infty \big\}, \\
\T_2
&= \big\{ \{(a_m,\eta_m)\} \colon 
a_m = A_1 + A_2(3/5)^m + a_m', \\
&\hspace{10em}\|(25/3)^{m/2}a_m'\|_{\ell^2}, \ \|3^{m/2}\eta_m\|_{\ell^2} < \infty \big\},
\end{align*}
with their respective norms
\begin{align*}
&\|\{(a_m,\eta_m)\}\|_{\T_\infty} 
= |A_1| + |A_2| + \|5^ma_m'\|_{\ell^\infty} + \|3^m \eta_m\|_\Lip, \\
&\|\{(a_m,\eta_m)\}\|_{\T_2}^2 
\ = |A_1|^2 + |A_2|^2 + \|(25/3)^{m/2} a_m'\|_{\ell^2}^2 + \|3^{m/2} \eta_m\|_{\ell^2}^2.
\end{align*}
Clearly both trace norms satisfy the triangle inequality. Note that the defined norm $\|\cdot\|_{\T_2}$ makes $\T_2$ a Hilbert Space with the obvious inner product. Similarly, we define norms on $\DI$ and $\DT$ by 
\begin{align*}
&\|u\|_{\DI} 
= \|u\|_{L^\infty(\SG)} + \|\triangle u\|_{L^\infty(\SG)}, \\
&\|u\|_{\DT}^2 \ 
= \|u\|_{L^2(\SG)}^2 + \|\triangle u\|_{L^2(\SG)}^2.
\end{align*}
In the above definition, we could have replaced $\|\cdot\|_{L^2}^2$ term with $\|\cdot\|_{L^\infty}^2$, but that would not be a Hilbert Space norm.  

As suggested by the notation, our goal is to prove that $R$ maps $\DI$ and $\DT$ to their corresponding trace spaces. In Section \ref{extension}, we will show that the mapping is onto. 

\begin{theorem} [Trace Theorem] \label{trace T1}
\indent
\begin{enumerate}
\item
The restriction operator $R \colon \DI \to \T_\infty$ is bounded and 
\[
\|Ru\|_{\T_\infty} 
\leq C_1 \|u\|_{L^\infty(\SG)} + C_2 \|\triangle u\|_{L^\infty(\SG)}.
\]
\item
The restriction operator $R \colon \DT \to \T_2$ is bounded and 
\[
\|Ru\|_{\T_2} 
\leq C_1 \|u\|_{L^\infty(\SG)} + C_2 \|\triangle u\|_{L^2(\SG)}.
\]
\end{enumerate}
\end{theorem}

The proof of the theorem is technical and rather long, so we split the proof into multiple lemmas. Our primary tool will be the Green's formula. Given any function $u$ on $\SG$ for which $\triangle u$ exists, we can write 
\begin{equation} \label{trace eq1}
u(x) = \int_\SG G(x, y) \triangle u(y) \ dy + h(x),
\end{equation}
where $G(x, y)$ is the Green's function (the definition is given in Section \ref{green}) and $h$ is the harmonic function with boundary conditions $\left. h \right|_{V_0} = \left. u \right|_{V_0}$. We will use the Green's function to relate an arbitrary function to its restriction to $X$ and its normal derivatives on $X$. The derivations are digressive, so we have placed these computations into their own section. The important formulas and inequalities are given by (\ref{green eq4}), (\ref{green eq5}), and (\ref{green eq7}). Note that the definition of the function $\Psi_m$ is given in (\ref{green eq3}). 

Since it is easy to check the conditions for the harmonic function $h$ in (\ref{trace eq1}), let us do that first.

\begin{lemma} 
\label{trace L1}
If $h$ is harmonic, then $Rh \in \T_\infty$ and $Rh \in \T_2$ with
\begin{align} 
\label{trace eq2}
&\|Rh\|_{\T_\infty} 
= |u(q_0)| + \frac{1}{2}|u(q_1)+u(q_2)-2u(q_0)|, \\
\label{trace eq3}
&\|Rh\|_{\T_2}^2
\ = |u(q_0)|^2 + \frac{1}{4}|u(q_1)+u(q_2)-2u(q_0)|^2+\frac{1}{8}|u(q_1)-u(q_2)|^2.
\end{align}
\end{lemma}

\begin{proof}
If $h$ is harmonic, then $h$ is a linear combination of the constant function, the skew-symmetric harmonic function (with respect to $X$) and the symmetric harmonic function (with respect to $X$). Then 
\[
\begin{pmatrix}
u(q_0) \\ u(q_1) \\ u(q_2)
\end{pmatrix}
= 
A_1 
\begin{pmatrix}
1 \\ 1 \\ 1
\end{pmatrix}
+ 
A_2
\begin{pmatrix}
0 \\ 1 \\ 1
\end{pmatrix}
+
A_3
\begin{pmatrix}
\hspace{0.75em} 0 \\ -1 \\ \hspace{0.75em} 1
\end{pmatrix},
\]
where the coefficients are the coefficients $A_1$, $A_2$, and $A_3$ are the weights of the constant, symmetric and skew-symmetric functions respectively. Solving the system for $A_1, A_2, A_3$ in terms of $\left. u \right|_{V_0}$, we find 
\[
A_1 = u(q_0),
\ A_2 = \frac{1}{2}(u(q_1)+u(q_2)-2u(q_0)),
\ \text{and}
\ A_3 = \frac{1}{2}(u(q_1)-u(q_2)).
\]
On $X$, we see that
\begin{enumerate}
\item 
a constant function is constant with zero normal derivative. 
\item
a skew-symmetric harmonic function is zero with normal derivative $A_3/3^m$. 
\item
a symmetric harmonic function has values $A_2(3/5)^m$ with zero normal derivative. 
\end{enumerate}
Then $h(x_m) = A_1 + A_2 (3/5)^m$ and $\partial_n h(x_m) = A_3/3^m$.
\end{proof}

In the following lemma, we prove the bulk of the $\DI$ case. Proving the lemma directly from the Green's formula would be difficult, so we employ the following indirect method. For the function values of $u \in \DI$ on the vertical boundary, we prove an intermediate statement about the linear combination $5u(x_{m+1}) - 3u(x_m)$. We consider the linear combination $5u(x_{m+1})-3u(x_m)$ because the troublesome $\sum_{k=1}^m \Psi_k(1,2,2)$ term of (\ref{green eq5}) cancels out in the linear combination $5G(x_{m+1},y) - 3G(x_m,y)$. Then the intermediate result, coupled with a lemma from Section \ref{sequences}, will give us the desired statement, except for a few estimates which we prove without much trouble.

Likewise, for the normal derivatives of $u \in \DI$ on the vertical boundary, we prove an intermediate statement about the linear combination $3\eta_{m+1} - \eta_m$ because the troublesome $\sum_{k=1}^m 3^k \Psi_k (0,-1,1)$ term in (\ref{green eq7}) disappears in the linear combination. The intermediary result, combined with the proper lemma from Section \ref{sequences} and more bounding, yields the desired normal derivative estimate. 

\begin{lemma} \label{trace L2}
If $u \in \DI$ with $u = 0$ on $V_0$, then $Ru \in \T_\infty$ and 
\begin{equation} \label{trace eq4}
\|Ru\|_{\T_\infty} 
\leq C\|\triangle u\|_{L^\infty(\SG)}.
\end{equation}
\end{lemma}

\begin{proof} 
Suppose $u \in \DI$ with $Ru = \{(a_m,\eta_m)\}$. Using the Green's formula (Proposition \ref{green P1}) on $5a_{m+1}-3a_m$ and the equation for $G(x_m,y)$ given by (\ref{green eq5}), after some simplification, we obtain
\[
5a_{m+1}-3a_m 
= \frac{1}{10} \(\frac{3}{5}\)^m \int_\SG [3\Psi_{m+1}(3,1,1)-5\Psi_m(-1,1,1)] \triangle u \ dy  
\]
Then applying inequality (\ref{green eq4}) yields
\begin{align*}
&|5a_{m+1}-3a_m| \\
&\tab\leq \|\triangle u\|_{L^\infty} \frac{1}{10} \(\frac{3}{5}\)^m \int_\SG |3\Psi_{m+1}(3,1,1)-5\Psi_m(-1,1,1)| \ dy  \\
&\tab\leq \|\triangle u\|_{L^\infty} \frac{C}{5^m}. 
\end{align*}
Rearranging the above inequality yields 
\[
\|5^m(5a_{m+1}-3a_m)\|_{\ell^\infty} 
\leq C\|\triangle u\|_{L^\infty}.
\]
Lemma \ref{sequences L1} implies that $a_m = A (3/5)^m + a_m'$, where $A = \lim_{m \to \infty} (5/3)^m a_m$ and 
\[
\|5^m a_m'\|_{\ell^\infty} 
\leq \|5^m(5a_{m+1}-3a_m)\|_{\ell^\infty}.
\]
The previous two inequalities immediately yield
\begin{equation} \label{trace eq1.1}
\|5^m a_m'\|_{\ell^\infty} 
\leq  C\|\triangle u\|_{L^\infty}.
\end{equation}
Since $a_m=u(x_m)=\int_\SG G(x_m,y)\triangle u(y)\ dy$, we have
\[
|a_m|\leq \|\triangle u\|_{L^\infty} \int_\SG |G(x_m,y)|\ dy.
\]
However, it follows from (\ref{green eq4}) and (\ref{green eq5}) that $\int_\SG G(x_m,y)\ dy \leq C(3/5)^m$, so 
\[
\(\frac{5}{3}\)^m |a_m| \leq C\|\triangle u\|_{L^\infty}
\]
Since $A = \lim_{m \to \infty} (5/3)^m a_m$, the above implies that 
\begin{equation} \label{trace eq1.2}
|A| \leq C \|\triangle u\|_{L^\infty}.
\end{equation}
We use a similar technique to prove the desired statement about the normal derivatives. Using the equation for $\eta_m$ given by (\ref{green eq7}) to compute $3\eta_{m+1}-\eta_m$, we obtain 
\[
3\eta_{m+1}-\eta_m 
= \frac{1}{10}\int_\SG [-3\Psi_{m+1}(5,1,-1)+5\Psi_m(1,-1,1)] \triangle u \ dy - 3\varphi_{m+1} + \varphi_m,
\]
where $\varphi_m$ was defined in the lemma. Then 
\begin{align*}
&|3\eta_{m+1}-\eta_m| \\
&\tab\leq C \|\triangle u\|_{L^\infty} \int_\SG |3\Psi_{m+1}(5,1,-1)-5\Psi_m(1,-1,1)| \ dy + |3\varphi_{m+1} - \varphi_m| \\
&\tab\leq C \|\triangle u\|_{L^\infty} \frac{1}{3^m},
\end{align*}
where we used (\ref{green eq4}) and (\ref{green eq1}) to bound the first and second terms respectively. Rearranging, we find that
\[
\|3^m(3\eta_{m+1}-\eta_m)\|_{\ell^\infty}
\leq C \|\triangle u\|_{L^\infty}.
\]
The above estimate allows us to apply Lemma \ref{sequences L2} which gives us 
\[
\|3^m\eta_m\|_\Lip = \|3^m(3\eta_{m+1}-\eta_m)\|_{\ell^\infty}.
\]
The previous two inequalities imply
\begin{equation} \label{trace eq1.3}
\|3^m\eta_m\|_\Lip 
\leq C\|\triangle u\|_{L^\infty}.
\end{equation} 
Finally, combining our inequalities (\ref{trace eq1.1}), (\ref{trace eq1.2}) and (\ref{trace eq1.3}), we see that
\[
\|Ru\|_{\T_\infty} 
= |A| + \|5^m a_m'\|_{\ell^\infty} + \|3^m \eta_m\|_\Lip
\leq C\|\triangle u\|_{L^\infty}.
\]
Since $a_m = A(3/5)^m + a_m'$ and $\|Ru\| < \infty$, we conclude that $Ru \in \T_\infty$. 
\end{proof}

In the following lemma, we prove the majority of the $\DT$ statement of the Trace Theorem. We use an indirect approach similar to that of the proof for the $\DI$ case, except the statements are considerably harder to prove. Proving the lemma directly from the Green's formula without proving the intermediary result would be extremely difficult, mainly because the Cauchy-Schwarz inequality is too wasteful for the type of estimate we desire.  

The outline of the proof is similar to that of Lemma \ref{trace L2}. For $u \in \DT$, we prove intermediary results about the linear combinations $5a_{m+2} - 8a_{m+1} + 3a_m$ and $3\eta_{m+1} - 16\eta_{m+1} + 5\eta_m$, where as usual, $a_m = u(x_m)$ and $\eta_m = \partial_n u(x_m)$. These linear combinations are written as linear combinations of integrals, but the primary integrand of each linear combination is supported on a set not containing $q_0$. This support allows us give a more precise estimate, thereby limiting the wastefulness of Cauchy-Schwartz. Then applying results from Section \ref{series} and some more bounding will give us the desired statements.

\begin{lemma} \label{trace L3}
If $u \in \DT$ with $u=0$ on $V_0$, then $Ru \in \T_2$ and 
\begin{equation} \label{trace eq5}
\|Ru\|_{\T_2} 
\leq C \|\triangle u\|_{L^2(\SG)}.
\end{equation}
\end{lemma}

\begin{proof}
Suppose $u \in \DT$ with $Ru = \{(a_m,\eta_m)\}$. Using the Green's formula (Proposition \ref{green P1}) on $5a_{m+2}-8a_{m+1}+3 a_m$ and the equation for $G(x_m,y)$ given by (\ref{green eq5}), after much computation, we obtain
\[
5a_{m+2}-8a_{m+1}+3 a_m 
= \(\frac{3}{5}\)^m \int_\SG \G_m \triangle u \ dy,
\]
where we defined
\[
\G_m(y) 
= \frac{1}{50} [9\Psi_{m+2}(3,1,1) - 20\Psi_{m+1}(1,0,0) + 25\Psi_m(1,-1,-1)].
\]
We show that $\G_m$ is supported on $D_m = Y_m \cup Y_{m+1} \cup Y_{m+2} \cup Z_m \cup Z_{m+1} \cup Z_{m+2}$. Since $\G_m$ is a linear combination of harmonic splines, we see that $\G_m$ vanishes on $Y_{m'} \cup Z_{m'}$ for $m' < m$. Using the harmonic extension algorithm, notice that
\begin{align*}
&25\Psi_m(1,-1,1)(y_{m+2}) = 25\Psi_m(1,-1,1)(z_{m+2}) = -9, \\
&20\Psi_{m+1}(1,0,0)(y_{m+2}) = 20\Psi_{m+1} (1,0,0)(z_{m+2}) = 0, \\
&9\Psi_{m+2}(3,1,1)(y_{m+2}) = 9\Psi_{m+2}(3,1,1)(z_{m+2}) = 9.
\end{align*}
Thus $\G_m(y_{m+2}) = \G_m(z_{m+2}) = 0$ and consequently, $\G_m$ vanishes on $Y_{m'} \cup Z_{m'}$ for $m' > m+2$, which proves that $\G_m$ is supported on $D_m$. Taking advantage of the support of $\G_m$, we can write 
\[
5a_{m+2}-8a_{m+1}+3 a_m 
= \(\frac{3}{5}\)^m \int_{D_m} \G_m \triangle u \ dy,
\]
Applying Cauchy-Schwarz and inequality (\ref{green eq4}) on the above equation yields
\[
|5a_{m+2}-8a_{m+1}+3 a_m|^2 
\leq C\|\triangle u\|_{L^2(D_m)}^2 \(\frac{3}{25}\)^m.
\]
By definition of $D_m$ and linearity of the integral, we have 
\begin{align}\label{trace eq2.1}
\begin{aligned} 
\|\triangle u\|^2_{L^2(D_m)} 
&= \sum_{k=m}^{m+2} \|\triangle u\|^2_{L^2(Y_k \cup Z_k)},
\\  
\|\triangle u\|^2_{L^2(\SG)} 
&= \sum_{k=1}^\infty \|\triangle u\|^2_{L^2(Y_k \cup Z_k)}.
\end{aligned}
\end{align}
Using the upper bound on $|5a_{m+2}-8a_{m+1}+3 a_m|^2$ and the above two equations, we obtain
\[
\|(25/3)^{m/2} (5a_{m+2}-8a_{m+1}+3 a_m)\|_{\ell^2}
\leq C\|\triangle u\|_{L^2(\SG)}.
\]
This estimate allows us to apply Lemma \ref{series L3}. Thus $a_m = A_1 + A_2 (3/5)^m + a_m'$, where $A_1 = \lim_{m\to\infty} a_m$, $A_2 = \lim_{m\to\infty}(5/3)^m (a_m-A_1)$, and
\[
\|(25/3)^{m/2} a_m'\|_{\ell^2}
\leq C \|(25/3)^{m/2} (5a_{m+2}-8a_{m+1}+3 a_m)\|_{\ell^2}.
\]
The above two inequalities immediately yield
\begin{equation} \label{trace eq2.2}
\|(25/3)^{m/2} a_m'\|_{\ell^2}
\leq C\|\triangle u\|_{L^2(\SG)}.
\end{equation}
We claim that $A_1=0$ and $|A_2| \leq C \|\triangle u\|_{L^2(\SG)}$. Applying Cauchy-Schwarz to the Green's formula for $a_m$, we find that that 
\[
\(\frac{5}{3}\)^m |a_m| 
\leq C \|\triangle u\|_{L^2(\SG)}.
\]
The above inequality implies that $A_1 = 0$ and
\begin{equation} \label{trace eq2.3}
|A_2| \leq C\|\triangle u\|_{L^2(\SG)}.
\end{equation}
We use a similar argument to prove the estimate on the normal derivatives. Using Lemma \ref{green L3} to compute $3\eta_{m+2}-16\eta_{m+1}+5\eta_m$, we see that 
\[
3\eta_{m+2}-16\eta_{m+1}+5\eta_m
= \int_\SG \Phi_m \triangle u \ dy - (3\varphi_{m+2} - 16\varphi_{m+1} + 5\varphi_m),
\]
where we defined
\[
\Phi_m 
= \frac{1}{10} [-3\Psi_{m+2}(5,1,-1)+10 \Psi_{m+1}(8,1,-1)-25\Psi_m(1,-1,1)].
\]
We show that $\Phi_m$ has support on $D_m$ as well. Since $\Phi_m$ is a linear combination of harmonic splines, $\Phi_m$ vanishes on $Y_{m'} \cup Z_{m'}$ for $m' < m$. Using the harmonic extension algorithm, we have
\begin{align*}
&-25\Psi_m(1,-1,1)(y_{m+2}) = 25\Psi_m(1,-1,1)(z_{m+2}) = 1, \\
&-10\Psi_{m+1}(8,1,-1)(y_{m+2}) = 10\Psi_{m+1}(8,1,-1)(z_{m+2}) = -2, \\
&-3\Psi_{m+2}(5,1,-1)(y_{m+1}) = 3\Psi_{m+2}(5,1,-1)(z_{m+1}) = -3.
\end{align*}
Thus, $\Phi_m(y_{m+2}) = \Phi_m(z_{m+2}) = 0$ and consequently, $\Phi_m$ vanishes on $Y_{m'} \cup Z_{m'}$ for $m' > m+2$. Using the compact support of $\Phi_m$, we can write
\[
3\eta_{m+2}-16\eta_{m+1}+5\eta_m
= \int_{D_m} \Phi_m \triangle u \ dy - (3\varphi_{m+2} - 16\varphi_{m+1} + 5\varphi_m),
\]
It is straightforward to find an upper bound on the linear combination of $\varphi_m$ terms. Using Cauchy-Schwarz and inequality (\ref{green eq2}), we obtain 
\[
|3\eta_{m+2}-16\eta_{m+1}+5\eta_m|^2 
\leq C \(|\varphi_{m+2}|^2 + |\varphi_{m+1}|^2 + |\varphi_m|^2\)
\leq C \|\triangle u\|_{L^2(D_m)}^2 \frac{1}{3^m}.
\]
Using Cauchy-Schwarz and inequality (\ref{green eq4}), we find that 
\[
\left| \int_{D_m} \Phi_m \triangle u \ dy \right|^2
\leq \|\triangle u\|_{L^2(D_m)}^2 \int_{D_m} |\Phi_m|^2 \ dy
\leq C \|\triangle u\|_{L^2(D_m)}^2 \frac{1}{3^m}. 
\]
Combining the above two inequalities and (\ref{trace eq2.1}) yields 
\begin{equation} \label{trace eq2.4}
\|3^{m/2} (3\eta_{m+2}-16\eta_{m+1}+5\eta_m)\|_{\ell^2}
\leq C\|\triangle u\|_{L^2(\SG)}. 
\end{equation}
The hypothesis of Lemma \ref{series L4} is satisfied, so we have $\eta_m = 5^m A + \eta_m'$ with 
\begin{equation} \label{trace eq2.5}
\|3^{m/2} \eta_m'\|
\leq C_1(\eta_2 - 5\eta_1)^2 + C_2\|3^m(3\eta_{m+2}-16\eta_{m+1}+5\eta_m)\|_{\ell^2}.
\end{equation}
However, applying Cauchy-Schwarz to (\ref{green eq7}) yields 
\[
|\eta_m| \leq C \|\triangle u\|_{L^2(\SG)} \frac{1}{3^{m/2}}.
\]
This forces $A = 0$ and so $\eta_m = \eta_m'$. Note that the above bound provides the upper bound $(\eta_2-5\eta_1)^2 \leq C\|\triangle u\|_{L^2(\SG)}^2$. Combining this inequality with (\ref{trace eq2.4}) and (\ref{trace eq2.5}) yields
\begin{equation} \label{trace eq2.6}
\|3^{m/2} \eta_m'\|^2 \leq C\|\triangle u\|_{L^2(\SG)}^2.
\end{equation}
Finally, using (\ref{trace eq2.2}), (\ref{trace eq2.3}) and (\ref{trace eq2.6}), we see that
\[
\|Ru\|_{\T_2}^2 
= |A_1|^2 + |A_2|^2 + \|(25/3)^{m/2}a_m'\|_{\ell^2}^2 + \|3^{m/2}\eta_m\|_{\ell^2}^2
\leq C \|\triangle u\|_{L^2(\SG)}^2.
\]
Since $a_m = A_2(3/5)^m + a_m'$ and $\|Ru\|_{\T_2}^2 < \infty$, we conclude that $Ru \in \T_2$.  
\end{proof}

Finally, we have the necessary results to prove the Trace Theorem.

\begin{proof} [Proof of the Trace Theorem]
Suppose $u \in \DI$ or $u \in \DT$, and $Ru = \{(a_m,\eta_m)\}$. Let $h$ be the harmonic function determined by the boundary values $\left. h \right|_{V_0} = \left. u \right|_{V_0}$. Let $w=u-h$, and note that $\triangle w = \triangle u$ and $w = 0$ on $V_0$. The Green's formula states that
\[
u(x) = h(x) + \int_\SG G(x, y) \triangle w(y) \ dy.
\]
\begin{enumerate}
\item
Suppose $u \in \DI$. Using triangle inequality on $u=w+h$, the estimate (\ref{trace eq2}) applied to $h$, and the estimate (\ref{trace eq4}) applied to $w$, we find that
\[
\|Ru\|_{\T_\infty}
\leq |u(q_0)| + \frac{1}{2}|u(q_1)+u(q_2)-2u(q_0)| + C \|\triangle u\|_{L^\infty(\SG)}.
\]
\item
Suppose $u \in \DT$. Using triangle inequality on $u=w+h$, (\ref{trace eq3}) applied to $h$, and (\ref{trace eq5}) applied to $w$, we find that
\begin{align*}
\|Ru\|_{\T_2}^2 
&\leq |u(q_0)|^2 + \frac{1}{4}|u(q_1)+u(q_2)-2u(q_0)| \\
&\hspace{5em} + \frac{1}{8}|u(q_1)-u(q_2)|^2+ C \|\triangle u\|_{L^2(\SG)}^2.
\end{align*}
\end{enumerate}
\end{proof}

\section{Extension Operators} \label{extension}

In this section, we present two different extension theorems. The first extension will be a right inverse to the restriction map $R$. The second extension will map solutions to differential equations on the half-gasket to a well-behaved function on the whole gasket. The ideas behind the two extensions are similar, but with different computations and formulas. In order to construct the desired extensions, we will require the following result. If will give us the exact conditions under which a piecewise function is in the domain of the Laplacian. 

\begin{proposition} [Gluing Theorem]
Let $u$ and $f$ be defined by gluing pieces $\{u_j\}$ and $\{f_j\}$ ($j = 0, 1, 2$), with $\triangle u_j = f_j$ on $F_j \SG$. Then $u \in \dom \triangle$ with $\triangle u = f$ if and only if $f_j(F_i q_j) = f_j(F_j q_i)$ ($i \not = j$) holds for $\{u_j\}$ and $\{f_j\}$ (so $u$ and $f$ are continuous) and the matching conditions on normal derivatives hold at the three points. 
\end{proposition}

\subsection{The Inverse Operator to $R$}

We seek a linear extension operator $E$ that is a right inverse of the restriction operator $R$. The desired extension will satisfy $E \colon \T_\infty \to \DI$ and $E \colon \T_2 \to \DT$. In order to construct this extension operator, we study piecewise biharmonic functions. Biharmonic functions satisfy the differential equation $\triangle^2 u = 0$ and in particular, biharmonic functions satisfying $\triangle u = C$ for some constant $C$ is a four-dimensional space on $\SG$. One way to specify a constant Laplacian function on $\SG$ is to specify the value of the function on $V_0$ and the constant.

\begin{lemma} 
\label{biharmonic L0}
Suppose $\triangle u = C$ on some cell of level $m$ with boundary points $p_0, p_1, p_2$. Then the outward normal derivative of $u$ at $p_j$ is
\begin{equation} 
\label{biharmonic eq2}
\partial_n u(p_j) 
= \(\frac{5}{3}\)^m [ 2u(p_j) - u(p_{j+1}) - u(p_{j-1}) ] + \frac{C}{3^{m+1}}.
\end{equation}
\end{lemma}

\begin{proof}
Let $v$ be the harmonic function on the cell with the boundary values $v(p_j) = 1$ and $v(p_{j+1}) = v(p_{j-1}) = 0$. Since $v$ is harmonic on a cell of level $m$, using (\ref{normal eq1}) with the proper normalization, we have $\partial_n v(p_j) = 2 (5/3)^m$ while $\partial_n v(p_{j+1}) = \partial_n v(p_{j-1}) = -(5/3)^m$. Applying the symmetric Gauss-Green formula (\ref{gauss-green}), we obtain the desired formula.
\end{proof}

\begin{lemma} 
\label{biharmonic L1}
Given any sequences $\{a_m\}$ and $\{\eta_m\}$, there exist a piecewise biharmonic function $u$ on $\SG$ and sequences $\{C_m'\}$ and $\{C_m\}$ such that $Ru = \{(a_m, \eta_m)\}$, $\triangle u = C_m'$ on $Y_m$, $\triangle u = C_m$ on $Z_m$, and the normal derivative matching conditions hold at $\{x_m\}$, $\{y_m\}$, and $\{z_m\}$. 
\end{lemma}

\begin{proof}
We construct two functions $u_1$ and $u_2$ such that $u_1(x_m)=a_m$ but $\partial_n u_1(x_m)=0$, while $u_2(x_m)=0$ but $\partial_n u(x_m)=\eta_m$. Then the sum $u=u_1+u_2$ will satisfy $Ru = \{(a_m,\eta_m)\}$. Of course, we must do this carefully so that $u$ satisfies the other claimed properties. 

Consider the symmetric piecewise biharmonic function $u_1$ satisfying $u_1(x_m) = a_m$, $u_1(y_m)=u_1(z_m)= (1/8)(5a_{m+1}+3a_m)$,
and $\triangle u_1 = D_m'$ on $Y_m \cup Z_m$ with 
\[
D_m' = 5^m \(\frac{3}{8}\) (5a_{m+1}-8a_m+3a_{m-1}).
\]
This information determines $u_1$ on $Y_m \cup Z_m$ because as mentioned earlier, a constant Laplacian function is determined by its boundary values and the value of its Laplacian. Consequently, $u_1$ is determined everywhere because $\SG = \bigcup_m (Y_m \cup Z_m)$. Using (\ref{biharmonic eq2}) to compute the normal derivatives of $u_1$ at $x_m$, $y_m$ and $z_m$, it is straightforward to check that $\partial_n u_1(x_m) = 0$ and the normal derivative matching conditions hold.

Consider the skew-symmetric piecewise biharmonic function $u_2$ satisfying the conditions $u_2(x_m) = 0$, $u_2(y_m) = -(1/8)(3/5)^m(\eta_{m+1} + \eta_m)$, $u_2(z_m) = - u_2(y_m)$, $\triangle u_2 = -E_m$ on $Y_m$ and $\triangle u_2 = E_m$ on $Z_m$, where 
\[
E_m = 3^m \(\frac{1}{8}\) (3 \eta_{m+1} - 16 \eta_m + 5 \eta_{m-1}).
\]
Again, these constraints determine $u_2$ everywhere on $\SG$. Writing down the normal derivatives of $u_2$ at $x_m$, $y_m$ and $z_m$ using (\ref{biharmonic eq2}), we see that $\partial_n u_2(x_m) = \eta_m$ and the normal derivative matching conditions hold.

Then the function $u=u_1+u_2$ satisfies $u(x_m) = a_m$, $\partial_n u(x_m) = \eta_m$,  
\begin{align}
\begin{aligned}
\label{biharmonic eq3}
u(y_m) 
&= \frac{1}{8} (5a_{m+1} + 3a_m) - \frac{1}{8} \(\frac{3}{5}\)^m (\eta_{m+1} + \eta_m), \\
u(z_m) 
&= \frac{1}{8} (5a_{m+1} + 3a_m) + \frac{1}{8} \(\frac{3}{5}\)^m (\eta_{m+1} + \eta_m),
\end{aligned}
\end{align}
$\triangle u = C_m'$ on $Y_m$ and $\triangle u = C_m$ on $Z_m$ where
\begin{align}
\begin{aligned}
\label{biharmonic eq4}
C_m' 
&= 5^m \(\frac{3}{8}\) (5a_{m+1} - 8a_m + 3a_{m-1}) - 3^m \(\frac{1}{8}\) (3 \eta_{m+1} - 16 \eta_m + 5 \eta_{m-1}), \\
C_m 
&= 5^m \(\frac{3}{8}\) (5a_{m+1} - 8a_m + 3a_{m-1}) + 3^m \(\frac{1}{8}\) (3 \eta_{m+1} - 16 \eta_m + 5 \eta_{m-1}).
\end{aligned}
\end{align}
Because normal derivatives add linearly, $u$ satisfies the normal derivative matching conditions at $x_m$, $y_m$ and $z_m$.
\end{proof}

As a result of the above lemma, we can define the extension operator $E$ which maps two sequences $\{(a_m,\eta_m)\}$ to the function $u$ given in the lemma. This operator is well defined because the process described by the lemma generates exactly one function for each pair of sequences. Additionally, it is not difficult to see that $E$ is a linear operator. 

\begin{theorem}
\label{biharmonic T1}
There exist a bounded linear extension map $E \colon \T_\infty \to \DI$ and $E \colon \T_2 \to \DT$ with $R \circ E = \text{Id}$.
\end{theorem} 

\begin{proof}
Suppose $\{(a_m,\eta_m)\} \in \T_\infty$ and let $u = E\{(a_m,\eta_m)\}$. In order to apply the Gluing Theorem, we need to check that $u$ is continuous. It suffices to check for continuity at $q_0$ because $u$ is clearly continuous everywhere else. In order to show that $u$ is continuous at $q_0$, we need to show that $\lim_{m \to \infty} u(x_m) = \lim_{m \to \infty} u(y_m) = \lim_{m \to \infty} u(z_m)$. Since $\{(a_m,\eta_m)\} \in \T_\infty$, we have $a_m = A_1 + A_2 (3/5)^m + a_m'$ with $\|5^m a_m'\|_{\ell^\infty} < \infty$ and $\|3^m \eta_m\|_\Lip < \infty$. Then (\ref{biharmonic eq3}) reads 
\begin{align*}
u(y_m) 
&= A_1 +  \frac{3}{4}\(\frac{3}{5}\)^m A_2 + \frac{1}{8} (5a_{m+1}' + 3a_m') - \frac{1}{8} \(\frac{3}{5}\)^m (\eta_{m+1} + \eta_m), \\
u(z_m) 
&= A_1 +  \frac{3}{4}\(\frac{3}{5}\)^m A_2 + \frac{1}{8} (5a_{m+1}' + 3a_m') + \frac{1}{8} \(\frac{3}{5}\)^m (\eta_{m+1} + \eta_m).
\end{align*}
Taking the limit $m \to \infty$ in the above equations, we see that $A_1 = \lim_{m \to \infty} u(y_m) = \lim_{m \to \infty} u(z_m) = \lim_{m \to \infty} a_m$, which verifies the continuity of $u$ at $q_0$. Recall that Lemma \ref{biharmonic L1} tells us that $u$ satisfies the normal derivative matching conditions at $\{x_m\}$, $\{y_m\}$ and $\{z_m\}$. Thus the hypotheses of the Gluing Theorem are satisfied, so the theorem implies that $\triangle u$ is well defined. We need to show that $\triangle u \in L^\infty(\SG)$. Observe that (\ref{biharmonic eq4}) reads
\begin{align*}
C_m' 
&= 5^m \(\frac{3}{8}\) (5a_{m+1}' - 8a_m' + 3a_{m-1}') - 3^m \(\frac{1}{8}\) (3 \eta_{m+1} - 16 \eta_m + 5 \eta_{m-1}), \\
C_m 
&= 5^m \(\frac{3}{8}\) (5a_{m+1}' - 8a_m' + 3a_{m-1}') + 3^m \(\frac{1}{8}\) (3 \eta_{m+1} - 16 \eta_m + 5 \eta_{m-1}).
\end{align*}
Using Lemma \ref{sequences L2} to obtain an upper bound on the normal derivative terms in $C_m$ and $C_m'$, we find that
\[
\|\triangle u\|_{L^\infty} 
\leq \|C_m\|_{\ell^\infty} + \|C_m'\|_{\ell^\infty}
\leq M_1 \|5^m a_m'\|_{\ell^\infty} + M_2\|3^m\eta_m\|_\Lip.
\]
Therefore, $E \colon \T_\infty \to \DI$. 

Suppose $\{(a_m,\eta_m)\} \in \T_2$ and let $u = E\{(a_m,\eta_m)\}$. Again, we need to check that $u$ is continuous at $q_0$ in order to apply the Gluing theorem. By definition of $\T_2$, we have $a_m = A_1 + A_2 (3/5)^m + a_m'$ with $\|(25/3)^{m/2} a_m'\|_{\ell^2} < \infty$ and $\|3^{m/2}\eta_m\|_{\ell^2} < \infty$. Then $|a_m'| \to 0$ and $|\eta_m| \to 0$. By the same argument for the $\T_\infty$ case, $u$ is continuous at $q_0$, hence continuous everywhere. By Lemma \ref{biharmonic L1}, $u$ satisfies the normal matching conditions at $\{x_m\}$, $\{y_m\}$ and $\{z_m\}$. Then $\triangle u$ is well defined by the Gluing Theorem. Finally, $\triangle u \in L^2(\SG)$ because 
\[
\|\triangle u\|_{L^2}^2
= \sum_{m=1}^\infty \frac{|C_m'|^2 + |C_m|^2}{3^m} 
\leq M_1 \sum_{m=1}^\infty \(\frac{25}{3}\)^m |a_m'|^2 + M_2 \sum_{m=1}^\infty 3^m |\eta_m|^2.
\]
Therefore, $E \colon \T_2 \to \DT$. 
\end{proof}

\subsection{Extensions of Solutions to Differential Equations on $\Omega$}

The material presented in this section is motivated by the classical theory of extending functions with $\triangle u \in L^p$ on a nice domain in Euclidean space $\R^n$ to functions with the same property on $\R^n$. We ask:
\begin{enumerate}
\item
Given $u \in \dom_{L^\infty} \triangle(\Omega)$, does there exist an extension $\u \in \DI$?
\item
Given $u \in \dom_{L^2} \triangle(\Omega)$, does there exist an extension $\u \in \DT$?
\end{enumerate}
We present two motivating examples before we proceed to the main extension results. 

\begin{theorem}
If $u$ is a harmonic function on $\Omega$, then either $u$ belongs to the two dimensional space of restrictions to $\Omega$ of even global harmonic functions on $\SG$, or the even extension of $u$ is not in $\dom\triangle$.
\end{theorem}

\begin{proof}
Let $\u$ denote the even extension of $u$. Then $\triangle \u=0$ on both $\Omega$ and its reflection. If $\u\in\dom\triangle$ then $\triangle\u$ must be a continuous function on $\SG$, hence identically zero, so $\u$ is an even global harmonic function. 
\end{proof}

\begin{theorem}
Suppose $u \in C(\overline \Omega)$ solves the BVP with $a_0 = C_1$ and $a_m = (2/3)(3/5)^m (C_1 + C_2)$ for some constants $C_1, C_2$. Then there exists a harmonic extension of $u$.  
\end{theorem}

\begin{proof}
Consider the harmonic function $\u$ on $\SG$ determined by the boundary values $\u(q_0) = 0$, $\u(q_1) = C_1$ and $\u(q_2) = C_2$. Simple computation shows that $\u(x_m) = (2/3)(3/5)^m (C_1 + C_2)$. Thus, $u = \u$ on $\overline \Omega$ and $\triangle \u = 0$, which shows that $\u$ is indeed a harmonic extension.
\end{proof}

In special cases, such as the one presented in the previous result, there exists a harmonic extension. In general, the desired extension will not be harmonic because the space of harmonic functions on $\SG$ is a three dimensional space so finding a harmonic extension $\u$ of $u$ satisfying the infinite number of conditions $\u(x_m) = a_m$ is unlikely. For that reason, we look for a piecewise biharmonic extension. In fact, this motivates our study of piecewise biharmonic functions to begin with. To prove the existence of an extension, we need the analogue of Lemma \ref{biharmonic L1}.

\begin{lemma} \label{extension L1}
Suppose $u \in \dom_{L^\infty} \triangle(\Omega)$ or $u \in \dom_{L^2} \triangle(\Omega)$. Then there exist a sequence $\{C_m\}$ and a piecewise biharmonic function $\u$ on $\SG$ satisfying $\u = u$ on $\overline \Omega$, $\triangle \u = C_m$ on $Z_m$, and the normal derivative matching conditions hold at $\{x_m\}$ and $\{z_m\}$. 
\end{lemma}

\begin{proof}
For convenience, we write $a_m=u(x_m)$ and $\eta_m = \partial_n u(x_m)$. Consider the function $\u = u$ on $\overline \Omega$, 
\begin{equation} 
\label{extension eq1}
\u(z_m) 
= \frac{1}{8} (5 a_{m+1} + 3a_m) + \frac{1}{8} \(\frac{3}{5}\)^m ( \eta_{m+1} + \eta_m),
\end{equation}
and $\triangle \u = C_m$ on $Z_m$ where
\begin{equation}
\label{extension eq2}
C_m 
= 5^m \(\frac{3}{8}\) (5 a_{m+1} - 8 a_m + 3 a_{m-1}) + 3^m \(\frac{1}{8}\) (3 \eta_{m+1} - 16 \eta_m + 5 \eta_{m-1}).
\end{equation}
For the same reason as before, these constraints completely determine $\u$ on $Z_m$. Hence we have defined a function $\u$ on $\SG$.

We claim that the normal matching conditions hold at $x_m$ and $z_m$. Using (\ref{biharmonic eq2}),
\begin{align*}
\leftarrow \partial_n \u(x_m) 
&= \(\frac{5}{3}\)^m \left[ 2\u(x_m) - \u(z_m) - \u(z_{m-1}) \right] + \frac{C_m}{3^{m+1}}, \\ 
\nwarrow \partial_n \u(z_m) 
&= \(\frac{5}{3}\)^m \left[ 2\u(z_m) - \u(z_{m-1}) - \u(x_m) \right] + \frac{C_m}{3^{m+1}}, \\
\searrow \partial_n \u(z_m) 
&= \(\frac{5}{3}\)^{m+1} \left[ 2\u(z_m) - \u(z_{m+1}) - \u(x_{m+1}) \right] + \frac{C_{m+1}}{3^{m+2}}.
\end{align*}
It is straightforward to check that our formulas for $\u(x_m)$, $\u(z_m)$, and $C_m$ imply the matching conditions hold at $\{x_m\}$ and $\{z_m\}$.
\end{proof}

The lemma allows us to define an extension operator. Let $E_\Omega$ be the extension operator that maps a function $u \in \dom_{L^\infty} \triangle(\Omega)$ or $u \in \dom_{L^2} \triangle(\Omega)$ to the function $E_\Omega u$ on $\SG$ as given in the lemma. This operator is well defined because for each $u$, there is exactly one $E_\Omega u$. It is clear that $E_\Omega$ is linear and that $E_\Omega u$ is continuous except possibly at $q_0$. 

\begin{theorem}
\label{extension T1}
Suppose $u \in \dom_{L^\infty}\triangle(\Omega)$. If $Ru \in \T_\infty$, then $E_\Omega u \in \DI$ and 
\[ 
\|\triangle (E_\Omega u)\|_{L^\infty(\SG)} 
\leq \|\triangle u\|_{L^\infty(\Omega)} + C\|Ru\|_{\T_\infty}. 
\]
The Trace Theorem implies the converse: if $E_\Omega u \in \DI$, then $Ru \in \T_\infty$.

\end{theorem} 

\begin{proof}
Suppose $u \in \dom_{L^\infty} \triangle(\Omega)$ and $Ru = \{(a_m, \eta_m)\} \in \T_\infty$. By definition of $\T_\infty$, we have $a_m = A_1 + A_2 (3/5)^m + a_m'$ with $\|5^m a_m'\|_{\ell^\infty} < \infty$ and $\|3^m \eta_m\|_\Lip < \infty$. We need to check that $E_\Omega$ is continuous at $q_0$. Observe that (\ref{extension eq1}) becomes
\[
E_\Omega u(z_m) = A_1 + A_2 \(\frac{3}{5}\)^m + \frac{1}{8} (5a_{m+1}' + 3a_m') + \frac{1}{8} \(\frac{3}{5}\)^m ( \eta_{m+1} + \eta_m).
\]
Taking the limit in the above equation, we see that $A_1 = \lim_{m\to\infty} a_m = \lim_{m\to\infty} E_\Omega u(z_m)$. This proves that $E_\Omega u$ is continuous. By Lemma \ref{extension L1}, the matching conditions for $u$ at $\{x_m\}$ and $\{z_m\}$ are satisfied. This allows us to apply the Gluing Theorem, and so $\triangle (E_\Omega u)$ exists. 

To prove that $E_\Omega u \in \DI$, observe that 
\[
\|5^m(5a_{m+1}-8a_m+3a_{m-1})\|_{\ell^\infty} 
\leq 16 \|5^m a_m'\|_{\ell^\infty}
\]
and by Lemma \ref{sequences L2}, 
\[
\|3^m(3\eta_{m+1}-16\eta_m+5 \eta_{m-1})\|_{\ell^\infty}
\leq 16\|3^m \eta_m\|_\Lip. 
\]
Using the above inequalities and the equation for $C_m$ given by (\ref{extension eq2}), we find that
\[
\|\triangle (E_\Omega u)\|_{L^\infty(\Omega')} 
= \max_m |C_m| 
\leq M_1 \|5^m a_m'\|_{\ell^\infty} + M_2 \|3^m \eta_m\|_\Lip.
\]
Then by triangle inequality, 
\[
\|\triangle (E_\Omega u)\|_{L^\infty(\SG)}
\leq \|\triangle u\|_{L^\infty(\Omega)} + M_1\|5^m a_m'\|_{\ell^\infty} + M_2\|3^m \eta_m\|_\Lip,
\]
which completes the proof.
\end{proof}

\begin{theorem} 
\label{extension T2}
Suppose $u \in \dom_{L^2}\triangle(\Omega)$. If $Ru \in \T_2$, then $E_\Omega u \in \DT$ and 
\[
\|\triangle (E_\Omega u)\|_{L^2(\SG)}^2
\leq \|\triangle u\|_{L^2(\Omega)}^2 + C \|Ru\|_{\T_2}^2. 
\]
The Trace Theorem implies the converse: if $E_\Omega u \in \DT$, then $Ru \in \T_2$.
\end{theorem}

\begin{proof}
Suppose $u \in \dom_{L^2}\triangle(\Omega)$ and $Ru = \{(a_m,\eta_m)\} \in \T_2$. By definition of $\T_2$, we know that $a_m = A_1 + A_2 (3/5)^m + a_m'$ with $\|(25/3)^{m/2} a_m'\|_{\ell^2} < \infty$ and $\|3^{m/2} \eta_m\|_{\ell^2} < \infty$. Then $|a_m'| \to 0$ and $|\eta_m| \to 0$. Using these limits, the same argument given in the proof of Theorem \ref{extension T1} shows that $E_\Omega u$ is continuous. Again, Lemma \ref{extension L1} guarantees the matching conditions for $u$ at $\{x_m\}$ and $\{z_m\}$ hold. The Gluing Theorem implies $\triangle (E_\Omega u)$ is well defined. 

To see why $E_\Omega u \in \DT$, we first see that 
\[
\|\triangle (E_\Omega u)\|_{L^2(\Omega')}^2 
= \sum_{m=1}^\infty \frac{1}{3^m}|C_m|^2
\leq M_1 \sum_{m=1}^\infty \(\frac{25}{3}\)^m |a_m'|^2 + M_2 \sum_{m=1}^\infty 3^m |\eta_m|^2.
\] 
Since $\|\triangle (E_\Omega u)\|_{L^2(\SG)}^2 = \|\triangle u\|_{L^2(\Omega)}^2 + \|\triangle (E_\Omega u)\|_{L^2(\Omega')}^2$, using the above inequality gives us
\[
\|\triangle (E_\Omega u)\|_{L^2(\SG)}^2 
\leq \|\triangle u\|_{L^2(\Omega)}^2 + M_1 \sum_{m=1}^\infty \(\frac{25}{3}\)^m |a_m'|^2 + M_2 \sum_{m=1}^\infty 3^m |\eta_m|^2.
\]
\end{proof}

We can interpret Theorem \ref{extension T1} and Theorem \ref{extension T2} by the following: $Ru \in \T_\infty$ is the minimal condition for extending an arbitrary function in $\dom_{L^\infty} \triangle(\Omega)$ to a function in $\DI$ and $Ru \in \T_2$ is the minimal condition for extending an arbitrary function in $\dom_{L^2} \triangle(\Omega)$ to a function in $\DT$. 

A function belonging to $\dom_{L^2}\triangle(\Omega)$ or $\dom_{L^\infty} \triangle(\Omega)$ is naturally a solution to the differential equation $\triangle u = f$ for $f \in L^2$ or $f \in L^\infty$ respectively. Solutions to this differential equation can be found using Theorem \ref{green T1}. 

As a special case of $E_\Omega$, we can extend harmonic functions $u$ on $\Omega$ provided that $Ru \in \T_2$ or $Ru \in \T_\infty$. Recall that the solution to this differential equation was explicitly given in Section \ref{solution}. The formula for the extended function will be given by (\ref{extension eq1}) and (\ref{extension eq2}), which can be simplified by using the normal derivative formula for harmonic functions (\ref{normal eq2}) and the recurrence relation (\ref{bvp recurrence}). 

\section{Appendix} \label{appendix}

\subsection{Green's Function Formulas}
\label{green}

For a given $m$ and a point $x \in V_m \setminus V_0$, let $\psi_x^m(y)$ denote the piecewise harmonic spline of level $m$ satisfying $\psi_x^m(y) = \delta_x(y)$ for $y \in V_m$ and extended harmonically for levels $m' > m$. Notice that $\psi_x^m \in \dom_0 \E$ because $x \not \in V_0$. 

\begin{proposition} [Green's Formula]
\label{green P1}
On $\SG$, the Dirichlet problem $- \triangle u = f$ on $\SG \setminus V_0$ and $u = 0$ on $V_0$ has a unique solution in $\dom \triangle$ for any continuous $f$, given by $u(x) = \int_{\SG} G(x, y) f(y) \ dy$ for the Green's function $G(x, y) = \lim_{M \to \infty} G_M(x, y)$ (uniform limit) where 
\[ 
G_M(x, y) 
= \sum_{k=1}^M \sum_{s, s' \in V_k \setminus V_{k-1}} \hspace{-0.5em} g(s, s') \psi_s^k\(x\) \psi_{s'}^k(y)
\]
and
\[
g(s, s') =
\begin{cases}
\ \vspace{0.5em} \frac{3}{10} \(\frac{3}{5}\)^k &\text{for } s = s' \in V_{k} \setminus V_{k-1}, \\
\ \frac{1}{10} \(\frac{3}{5}\)^k &\text{for } s, s' \in F_w K, \ |w| = k-1 \text{ and } s \not = s'.
\end{cases}
\]
\end{proposition}

From the Green's formula, we have the following simple observation. 

\begin{theorem}
\label{green T1}
Let $G(x,y)$ denote the Green's function on $\SG$. Let $G_\Omega(x,y) = G(x,y)-G(x,Ry)$ for $x,y\in\Omega$ where $R$ denotes the reflection. Then $G_\Omega$ is the Green's function for $\Omega$, namely
\[
u(x) = \int_\Omega G_\Omega(x,y)f(y)\ dy 
\]
solves $-\triangle u = f$ on $\Omega$ subject to $\left. u \right|_\Omega = 0$. 
\end{theorem}

To simplify notation, we drop the superscript $^m$ on functions of the form $\psi_{x_m}^m$, $\psi_{y_m}^m$, and $\psi_{z_m}^m$ because unless otherwise notated, the superscript index matches the subscript index. It follows immediately from the definition that 
\begin{equation} \label{green eq1}
\int_\SG |\psi_{x_m}| \ dy 
= \int_{\SG} |\psi_{y_m}| \ dy 
= \int_{\SG} |\psi_{z_m}| \ dy 
= \frac{2}{3^{m+1}}.
\end{equation}
Additionally, since $|\psi_{x_m}|^2 \leq |\psi_{x_m}|$, we have
\begin{equation} \label{green eq2}
\int_{\SG} |\psi_{x_m}|^2 \ dy 
= \int_{\SG} |\psi_{y_m}|^2 \ dy 
= \int_{\SG} |\psi_{z_m}|^2 \ dy 
\leq \frac{2}{3^{m+1}}.
\end{equation}
To further simply notation, define the function 
\begin{equation} \label{green eq3}
\Psi_m(a, b, c)(y) = a \psi_{x_m}(y) + b \psi_{y_m}(y) + c \psi_{z_m}(y).
\end{equation} Using (\ref{green eq1}) and (\ref{green eq2}), we have the estimates
\begin{equation} \label{green eq4}
\int_\SG |\Psi_m (a, b, c)| \ dy 
\leq \frac{C_1}{3^m}
\indent \text{and} \indent
\int_\SG |\Psi_m (a, b, c)|^2 \ dy 
\leq \frac{C_2}{3^m},
\end{equation}
for constants $C_1$ and $C_2$ depending only on $a, b, c$.

\begin{lemma} \label{green L1}
The Green's function evaluated at $x_m$ is
\begin{equation} \label{green eq5}
G(x_m, y) 
= \frac{2}{15} \(\frac{3}{5}\)^m \sum_{k=1}^m \Psi_k(1, 2, 2)(y) + \frac{1}{6} \(\frac{3}{5}\)^m \Psi_m(1, -1, -1)(y).
\end{equation}
\end{lemma}

\begin{proof}
Note the following observations:
\begin{enumerate}
\item
If $k > m$, then $\psi_s^k(x_m) = 0$. 
\item
If $k=m$, then $\psi_{x_m}(x_m) = 1$. If $k=m$ and $s \not = x_m$, then $\psi_s^m (x_m) = 0$.
\item 
If $k < m$ with $s \not = y_k$ and $s \not = z_k$, then $\psi_s^k(x_m) = 0$. 
\end{enumerate}
Using these facts, we have
\begin{align*}
G(x_m,y) 
&=\sum_{k=1}^{m-1}\sum_{s'\in V_k\setminus V_{k-1}}[g(y_k,s')\psi_{y_k}(x_m) + g(z_k,s')\psi_{z_k}(x_m)]\psi_{s'}^k(y)\\
&\tab+\sum_{s'\in V_m\setminus V_{m-1}} g(x_m,s')\psi_{s'}^m(y).
\end{align*}
Using the harmonic extension algorithm, for $k < m$, we have
\[
\psi_{y_k}(x_m) = \frac{2}{3} \(\frac{3}{5}\)^{m-k}
\tab \text{and} \tab
\psi_{z_k}(x_m) = \frac{2}{3} \(\frac{3}{5}\)^{m-k}.
\]
Since $g(s, s') = 0$ if $s$ and $s'$ are in different cells of level $k-1$, we deduce that
\begin{align*}
\sum_{s' \in V_k \setminus V_{k-1}} \left[ g(y_k, s') + g(z_k, s') \right] \psi_{s'}^k(y)
&= \frac{1}{5} \(\frac{3}{5}\)^k \Psi_k(1, 2, 2)(y), \\
\sum_{s' \in V_m \setminus V_{m-1}} g(x_m, s') \psi_{s'}^m (y)
&= \frac{1}{10} \(\frac{3}{5}\)^m \Psi_m(3, 1, 1)(y).
\end{align*}
Substituting these equations into the most recent equation for $G(x_m, y)$ completes the proof. 
\end{proof}

\begin{lemma} \label{green L2}
The Green's function evaluated at $z_m$ is
\begin{equation} \label{green eq6}
G(z_m, y) 
= \frac{1}{10} \(\frac{3}{5}\)^m \sum_{k=1}^m \Psi_k(1, 2, 2)(y) + \frac{1}{10} \(\frac{1}{5^m}\) \sum_{k=1}^m 3^k \Psi_k(0, -1, 1)(y).
\end{equation}
\end{lemma}

\begin{proof}
We use a similar process to find the formula for $G(z_m, y)$. Note the following observations:
\begin{enumerate}
\item
If $k > m$, then $\psi_s^k(z_m) = 0$. 
\item
If $k = m$, then $\psi_{z_m}(z_m) = 1$. If $k = m$ and $s \not = z_m$, then $\psi_s^m(z_m) = 0$.
\item
If $k < m$ with $s \not = y_k$ and $s \not = z_k$, then $\psi_s^k(z_m) = 0$. 
\end{enumerate}
Using these facts, we have
\begin{align*}
G(z_m, y) 
&=\sum_{k=1}^{m-1} \sum_{s' \in V_k \setminus V_{k-1}}[g(y_k,s')\psi_{y_k}(z_m)+g(z_k,s')\psi_{z_k}(z_m)]\psi_{s'}^k(y)\\
&\tab +\sum_{s'\in V_m\setminus V_{m-1}} g(z_m, s')\psi_{s'}^m(y).
\end{align*}
Using the harmonic algorithm, for $k < m$, we have  
\begin{align*}
\psi_{y_k}(z_m) 
&= \frac{1}{2} \(\frac{3}{5}\)^{m-k} - \frac{1}{2} \(\frac{1}{5}\)^{m-k}
\psi_{z_k}(z_m) \\  
&= \frac{1}{2} \(\frac{3}{5}\)^{m-k} + \frac{1}{2} \(\frac{1}{5}\)^{m-k}.
\end{align*}
Since $g(s, s') = 0$ if $s$ and $s'$ are in different cells of level $k - 1$, we deduce that
\begin{align*}
\sum_{s' \in V_k \setminus V_{k-1}} g(y_k, s') \psi_{s'}^k (y)
&= \frac{1}{10} \(\frac{3}{5}\)^k \Psi_k(1, 3, 1)(y), \\
\sum_{s' \in V_k \setminus V_{k-1}} g(z_k, s') \psi_{s'}^k (y)
&= \frac{1}{10} \(\frac{3}{5}\)^k \Psi_k(1, 1, 3)(y), \\
\sum_{s' \in V_m \setminus V_{m-1}} g(z_m, s') \psi_{s'}^m (y)
&= \frac{1}{10} \(\frac{3}{5}\)^m \Psi_m(1, 1, 3)(y).
\end{align*}
Making these substitutions into the previous equation for $G(z_m, y)$ completes the proof.
\end{proof}

\begin{lemma} \label{green L3}
If  $u = 0$ on $V_0$ and $\triangle u$ exists on $\SG$, then
\begin{equation} \label{green eq7}
\begin{aligned}
\partial_n u(x_m)
&=\frac{3}{5} \(\frac{1}{3^m}\) \sum_{k=1}^m 3^k \int_\SG \Psi_k(0, -1, 1) \triangle u \ dy \\
&\tab - \frac{1}{2} \int_\SG \Psi_m(1, -1, 1) \triangle u \ dy - \varphi_m,
\end{aligned}
\end{equation}
where $\varphi_m = \int_{Z_m} \psi_{x_m} \triangle u \ dy$.
\end{lemma}

\begin{proof}
Let $v$ be the harmonic function on $Z_m$ determined by the boundary values $v(x_m) = 1$ and $v(z_{m-1}) = v(z_m) = 0$. Note that $v = \psi_{x_m}$ on $Z_m$. Since $Z_m$ is a cell of level $m$ and $v$ is harmonic, using (\ref{normal eq1}) with the proper normalization constant, we have $\leftarrow \partial_n v(x_m) = 2 (5/3)^m$ and $\searrow \partial_n v(z_{m-1}) = \ \nwarrow \partial_n v(z_m) = -(5/3)^m$. These equations, together with the symmetric Gauss-Green formula (\ref{gauss-green}) applied to the functions $u$ and $v$, yield
\[
\leftarrow \partial_n u(x_m) 
= \int_{Z_m} \psi_{x_m} \triangle u \ dy + \(\frac{5}{3}\)^m \left[ 2u(x_m) - u(z_m) - u(z_{m-1}) \right].
\]
Using the Green's formula, the formulas for $G(x_m, y)$ and $G(z_m, y)$ given by (\ref{green eq5}) and (\ref{green eq6}) respectively, and the normal derivative matching condition at $x_m$ yields the desired formula. 
\end{proof}

\subsection{Lemmas for Sequences}
\label{sequences}

\begin{lemma} \label{sequences L1}
Given a sequence $\{a_m\}$, $\|5^m(5a_{m+1}-3a_m)\|_{\ell^\infty} < \infty$ if and only if $a_m = A (3/5)^m + a_m'$ with $\|5^m a_m'\|_{\ell^\infty} < \infty$. Furthermore,
\[
\|5^m a_m'\|_{\ell^\infty} \leq \|5^m(5a_{m+1}-3a_m)\|_{\ell^\infty}.
\]
Note that the equation for $a_m$ and the bound for $a_m'$ implies $A = \lim_{m\to\infty} (5/3)^m a_m$.
\end{lemma}

\begin{proof}
Clearly the second statement implies the first statement. Conversely, making the substitution $d_m = (5/3)^m a_m$, we find that
\[
3\|3^m (d_{m+1} - d_m)\|_{\ell^\infty} 
= \|5^m(5a_{m+1}-3a_m)\|_{\ell^\infty} 
< \infty.
\]
This inequality implies that $\{d_m\}$ is a Cauchy sequence and by completeness of the reals, $d_m \to D$ for some $D$. Then $a_m = (3/5)^m D + (3/5)^m (d_m-D)$. Writing $d_m$ as a telescoping series
\[
d_m = D + \sum_{k=m}^\infty (d_k-d_{k+1})
\]
and using the inequality $\|3^m (d_{m+1}-d_m)\|_{\ell^\infty} < \infty$, we obtain
\[
|d_m - D| 
\leq \sum_{k=m}^\infty |d_k-d_{k+1}| 
\leq \frac{1}{3^m} \|5^m(5a_{m+1}-3a_m)\|_{\ell^\infty}.
\]
Then defining $a_m' = (3/5)^m(d_m-D)$, we see that
\[
\|5^m a_m'\|_{\ell^\infty}  
= \|3^m(d_m-D)\|_{\ell^\infty}
\leq \|5^m(5a_{m+1}-3a_m)\|_{\ell^\infty}.
\]
\end{proof}

\begin{lemma} \label{sequences L2}
Given a sequence $\{\eta_m\}$, $\|3^m (3\eta_{m+1}-\eta_m)\|_{\ell^\infty} < \infty$ if and only if $\|3^m \eta_m\|_\Lip < \infty$. In fact, 
\[
\|3^m (3\eta_{m+1}-\eta_m)\|_{\ell^\infty} 
= \|3^m\eta_m\|_\Lip.
\] 
\end{lemma}

\begin{proof}
If $\|3^m (3\eta_{m+1}-\eta_m)\|_{\ell^\infty} < \infty$, then 
\[
\|3^m \eta_m\|_\Lip 
= \sup_m 3^m |3\eta_{m+1}-\eta_m| 
= \|3^m (3\eta_{m+1}-\eta_m)\|_{\ell^\infty}
< \infty.
\]
Conversely, if $\|3^m \eta_m\|_\Lip < \infty$, then
\[
3^m |3 \eta_{m+1}-3\eta| 
= \|3^{m+1}\eta_{m+1}-3^m \eta_m\| 
\leq \|3^m \eta_m\|_\Lip
< \infty.
\]
\end{proof}

\subsection{Lemmas for Series}
\label{series}

\begin{lemma} \label{series L1}
Fix a constant $r<1$ and a sequence $\{a_m\}$. Then $\|r^{m/2}a_m\|_{\ell^2}<\infty$ if and only if $\|r^{m/2}(a_{m+1}-a_m)\|_{\ell^2}<\infty$. More specifically, 
\[
\|r^{m/2}a_m\|_{\ell^2} 
\leq C_1|a_1|^2 + C_2 \|r^{m/2}(a_{m+1}-a_m)\|_{\ell^2}.
\]
\end{lemma}

\begin{proof}
The first statement obviously implies the second statement. Conversely, writing $a_m$ as a telescoping series
\[
a_m 
= a_1 + \sum_{k=1}^{m-1} (a_{k+1} - a_k) 
= a_1 + \sum_{k=1}^{m-1} (a_{m-k+1} - a_{m-k}),
\]
we see that
\[
r^{m/2} a_m 
= r^{m/2} a_1 + \sum_{k=1}^{m-1} (a_{m-k+1} - a_{m-k}) r^{(m-k)/2}r^{k/2}.
\]
Using Minkowski's inequality, we have
\begin{align*}
&\left\|\sum_{k=1}^{m-1}(a_{m-k+1}-a_{m-k})r^{(m-k)/2}r^{k/2}\right\|_{\ell^2}\\
&\tab\leq \sum_{k=1}^\infty r^{k/2} \left\|(a_{m-k+1}-a_{m-k}) r^{(m-k)/2}\chi_{k<m}\right\|_{\ell^2} \\
&\tab\leq \sum_{k=1}^\infty r^{k/2} \|(a_{m+1} - a_m) r^{m/2}\|_{\ell^2}.
\end{align*}
Using Minkowski's inequality again and the above inequality, we find that
\[
\|r^{m/2} a_m\|_{\ell^2} 
\leq \|r^{m/2} a_1\|_{\ell^2} +\left\|\sum_{k=1}^{m-1}(a_{m-k+1}-a_{m-k}) r^{(m-k)/2}r^{k/2}\right\|_{\ell^2},
\]
which completes the proof. 
\end{proof}

\begin{lemma} \label{series L2}
Fix a constant $r>1$ and a sequence $\{a_m\}$. Then $a_m = A + a_m'$ with $\|r^{m/2} a_m'\|_{\ell^2} < \infty$ if and only if $\|r^{m/2} (a_{m+1}-a_m)\|_{\ell^2} < \infty$. In fact, 
\[
\|r^{m/2} a_m'\|_{\ell^2} 
\leq C \|r^{m/2} (a_{m+1}-a_m)\|_{\ell^2}.
\]
\end{lemma}

\begin{proof}
Clearly, the first statement implies the second statement. To prove the converse, we first show that $\{a_m\}$ is Cauchy. For $m > n$, we have
\[
a_m - a_n = \sum_{k=n}^{m-1} \(a_{k+1} - a_k\) r^{k/2}r^{-k/2}
\]
and applying Cauchy-Schwarz yields 
\[
|a_m - a_n| 
\leq \(\sum_{k=n}^{m-1} (a_{k+1} - a_k)^2 r^k \)^{1/2} \(\sum_{k=n}^{m-1} \frac{1}{r^k} \)^{1/2} \leq C \sqrt{\frac{1}{r^n}}.
\]
It follows that $\{a_m\}$ is Cauchy and by completeness of the reals, $a_m \to A$ for some $A$. Since
\[
a_m - A
= \sum_{k=m}^\infty (a_k - a_{k+1})
= \sum_{k=0}^\infty (a_{m+k} - a_{m+k+1}),
\]
we see that
\[
r^{m/2} (a_m - A) 
= \sum_{k=0}^\infty r^{(m+k)/2}r^{-k/2}(a_{m+k} - a_{m+k+1}).
\]
Using this equation and Minkowski's inequality, we have
\begin{align*}
\|r^{m/2}(a_m-A)\|_{\ell^2}
&\leq \sum_{k=0}^\infty r^{-k/2} \|(a_{k+m} - a_{k+m+1}) r^{(k+m)/2}\|_{\ell^2} \\
&\leq\sum_{k=0}^\infty r^{-k/2} \|(a_m - a_{m+1})r^{m/2}\|_{\ell^2},
\end{align*}
which completes the proof.
\end{proof}

\begin{lemma} \label{series L3}
Given a sequence $\{a_m\}$, $\|(25/3)^{m/2} (5 a_{m+2} - 8a_{m+1} + 3a_m)\|_{\ell^2} < \infty$ if and only if $a_m = A_1 + A_2 (3/5)^m + a_m'$ with $\|(25/3)^{m/2} a_m'\|_{\ell^2} < \infty$. More specifically,
\[
\| (25/3)^{m/2} a_m'\|_{\ell^2} 
\leq C \|(25/3)^m (5 a_{m+2} - 8a_{m+1} + 3a_m)\|_{\ell^2}.
\]
Note that the equation for $a_m$ and the bound for $a_m'$ imply that $A_1 = \lim_{m\to\infty} a_m$ and $A_2 = \lim_{m\to\infty} (5/3)^m(a_m- A_1)$. 
\end{lemma}

\begin{proof}
Clearly the second statement implies the first statement. To prove the converse, we apply Lemma \ref{series L2} twice. Making the substitution $3^m d_m = 5^m (a_{m+1}-a_m)$ yields 
\[
\sum_{m=1}^\infty \(\frac{25}{3}\)^m (5a_{m+2}-8a_{m+1}+3a_m)^2
= 9 \sum_{m=1}^\infty 3^m (d_{m+1}-d_m)^2
< \infty.
\]
The hypotheses of the lemma are satisfied for $\{d_m\}$, so we have $d_m = D + d_m'$ with 
\[
\sum_{m=1}^\infty 3^m |d_m'|^2 
\leq C \sum_{m=1}^\infty 3^m (d_{m+1}-d_m)^2.
\] 
In order to apply the lemma again, define $e_m = a_m + (5/2)(3/5)^m D$ so that
\[
\sum_{m=1}^\infty 3^m |d_m'|^2 
= \sum_{m=1}^\infty \(\frac{25}{3}\)^m (e_{m+1}-e_m)^2
< \infty.
\]
Using the lemma again, except on the sequence $\{e_m\}$, we have $e_m = E + e_m'$ with the estimate 
\[
\sum_{m=1}^\infty \(\frac{25}{3}\)^m |e_m'|^2 
\leq C \sum_{m=1}^\infty \(\frac{25}{3}\)^m (e_{m+1}-e_m)^2.
\]
Finally, using the definition of $e_m$, we find that 
$a_m = E - (5/2)(3/5)^m D + e_m'$. Combining the above equations and inequalities, we obtain 
\[
\sum_{m=1}^\infty \(\frac{25}{3}\)^m |e_m'|^2 
\leq C \sum_{m=1}^\infty \(\frac{25}{3}\)^m (5a_{m+2}-8a_{m+1}+3a_m)^2.
\] 
\end{proof}

\begin{lemma} \label{series L4}
Given a sequence $\{\eta_m\}$, $\|3^{m/2}(3\eta_{m+2}-16\eta_{m+1}+5\eta_m)\|_{\ell^2} < \infty$ if and only if $\eta_m = 5^m A + \eta_m'$ with $\|3^{m/2} \eta_m'\|_{\ell^2} < \infty$. Furthermore,
\[
\|3^{m/2} \eta_m'\|_{\ell^2}^2
\leq C_1(\eta_2 - 5\eta_1)^2 + C_2\|3^{m/2}(3\eta_{m+2}-16\eta_{m+1}+5\eta_m)\|_{\ell^2}^2.
\]
\end{lemma}

\begin{proof}
The second statement obviously implies the first statement. To prove the converse, we use both Lemma \ref{series L1} and Lemma \ref{series L2}. Define $e_m = 3^m (\eta_{m+1} - 5 \eta_m)$ so that
\[
\sum_{m=1}^\infty 3^m(3\eta_{m+2}-16\eta_{m+1}+5\eta_m)^2 
= \sum_{m=1}^\infty \frac{1}{3^m}(e_{m+1}-e_m)^2 
< \infty.
\]
Applying Lemma \ref{series L1} to the sequence $\{e_m\}$ gives us
\[
\sum_{m=1}^\infty \frac{1}{3^m} |e_m|^2 
\leq C_1|e_1|^2 + C_2 \sum_{m=1}^\infty 3^m (3\eta_{m+2}-16\eta_{m+1}+5\eta_m)^2 
< \infty.
\]
Making the substitution $5^m d_m = \eta_m$, we see that 
\[
\sum_{m=1}^\infty \frac{1}{3^m} |e_m|^2 
= \sum_{m=1}^\infty 3^m (\eta_{m+1} - 5 \eta_m)^2 
= 25 \sum_{m=1}^\infty 75^m (d_{m+1} - d_m)^2 
< \infty.
\]
Applying Lemma \ref{series L2} to the sequence $\{d_m\}$, we find that $d_m = D + d_m'$ with 
\[
\sum 75^m |d_m'|^2 \leq C \sum_{m=1}^\infty 75^m (d_{m+1}-d_m)^2.
\]
It follows from the definition of $d_m$ that $\eta_m = 5^m D + 5^m d_m'$. Defining $\eta_m' = 5^m d_m'$ and combining the above equations and inequalities, we obtain
\[
\sum_{m=1}^\infty 3^m |\eta_m'|^2 
\leq C_1(\eta_2 - 5\eta_1)^2 + C_2 \sum_{m=1}^\infty 3^m (3\eta_{m+2}-16\eta_{m+1}+5\eta_m)^2.
\] 
\end{proof}

\section{Acknowledgements}

We are grateful to the referee for suggesting the statement and proof of Theorem \ref{new}. 


\end{document}